\documentclass{amsart}

\usepackage{amsmath,amsfonts,amssymb,amsthm,mathtools}

\usepackage{verbatim}
\usepackage{amsmath,amsfonts}
\usepackage[mathscr]{euscript} 
\usepackage{amsthm}
\usepackage{url}

\usepackage{graphicx} 

\usepackage{enumitem} 

\usepackage{hyperref} 
\hypersetup{colorlinks} 

\usepackage[colorinlistoftodos]{todonotes}

\RequirePackage{cleveref}
\usepackage{hypcap}
\hypersetup{colorlinks=true, citecolor=darkblue, linkcolor=darkblue}
\definecolor{darkblue}{rgb}{0.0,0,0.7}
\newcommand{\darkblue}{\color{darkblue}}
\newcommand{\defn}[1]{\emph{\darkblue #1}}

\setlist[enumerate]{
	label=\textnormal{({\roman*})},
	ref={\roman*}}

\makeatletter
\def\th@plain{%
	\thm@notefont{}
	\itshape 
}
\def\th@definition{%
	\thm@notefont{}
	\normalfont 
}
\makeatother

\makeatletter
\def\fdsy@scale{1}
\newcommand\fdsy@mweight@normal{Book}
\newcommand\fdsy@mweight@small{Book}
\newcommand\fdsy@bweight@normal{Medium}
\newcommand\fdsy@bweight@small{Medium}

\DeclareFontFamily{U}{FdSymbolB}{}
\DeclareFontShape{U}{FdSymbolB}{m}{n}{
	<-7.1> s * [\fdsy@scale] FdSymbolB-\fdsy@mweight@small
	<7.1-> s * [\fdsy@scale] FdSymbolB-\fdsy@mweight@normal
}{}
\DeclareFontShape{U}{FdSymbolB}{b}{n}{
	<-7.1> s * [\fdsy@scale] FdSymbolB-\fdsy@bweight@small
	<7.1-> s * [\fdsy@scale] FdSymbolB-\fdsy@bweight@normal
}{}

\makeatother

\DeclareSymbolFont{fdrelations}{U}{FdSymbolB}{m}{n}
\SetSymbolFont{fdrelations}{bold}{U}{FdSymbolB}{b}{n}

\DeclareMathSymbol{\lescc}{\mathrel}{fdrelations}{66}

\newtheorem{thm}{Theorem}[section]
\newtheorem{lemma}[thm]{Lemma}

\newtheorem{cor}[thm]{Corollary}
\newtheorem{prop}[thm]{Proposition}
\newtheorem{conj}[thm]{Conjecture}
\newtheorem{conv}[thm]{Convention}

\theoremstyle{definition}
\newtheorem{ex}[thm]{Example}
\newtheorem{rem}[thm]{Remark}

\numberwithin{figure}{section}
\numberwithin{equation}{section}

\def\wh{\widehat}

\def\emp{\nothing}
\def\sq{\square}
\def\zz{\mathbb Z}
\def\nn{\mathbb N}

\def\pp{\mathbb P}

\def\Om{\Omega}

\def\la{\lambda}
\def\ga{\gamma}

\def\al{\alpha}
\def\be{\beta}

\def\ve{\varepsilon}

\def\vk{\varkappa}

\def\cF{\mathcal F}

\def\ssu{\subset}

\def\<{\langle}
\def\>{\rangle}

\def\oa{\overrightarrow}

\def\rT{{\text {\rm T} } }

\def\0{{\mathbf 0}}

\def\nothing{\varnothing}

\def\.{\hskip.06cm}
\def\ts{\hskip.03cm}

\def\nin{\noindent}

\def\SP{{\textsc{\#P}}}

\def\GapP{{\textsc{GapP}}}
\def\poly{{\textsc{P}}}

\def\aN{\textrm{N}}
\def\aNr{\textrm{\em N}}
\def\aF{\textrm{F}}
\def\aFr{\textrm{\em F}}

\def\aK{\textrm{K}}
\def\aKr{\textrm{\em K}}

\def\ana{\emph{\textsf{a}}}
\def\bnb{\emph{\textsf{b}}}

\def\A{A}
\def\B{B}
\def\C{C}
\def\D{D}

\DeclareMathOperator{\Cen}{\textnormal{C}} 
\DeclareMathOperator{\Cendown}{\textnormal{C}_{\textnormal{down}}} 
\DeclareMathOperator{\Cenup}{\textnormal{C}_{\textnormal{up}}} 
\def\cN{\mathcal N}
\DeclareMathOperator{\Ec}{\mathcal{E}} 

\def\Forbdown{F_{\textnormal{down}}} 
\def\Forbup{F_{\textnormal{up}}} 

\newcommand{\Kc}{\mathcal{K}} 

\DeclareMathOperator{\Rb}{\mathbb{R}} 

\DeclareMathOperator{\Reg}{\textnormal{Reg}} 

\def\precc{\prec}
\def\succc{\succ}

\DeclareMathOperator{\vb}{\mathbf{v}} 
\DeclareMathOperator{\vmax}{v_{\textnormal{max}}} 
\DeclareMathOperator{\vmin}{v_{\textnormal{min}}} 

\DeclareMathOperator{\wgt}{\mathtt{wt}} 
\DeclareMathOperator{\Zb}{\mathbb{Z}} 
\DeclareMathOperator{\zero}{\mathbf{0}} 

\def\Cr{\mathcal{C}}
\def\Yu{{Y^{\<u\>}}}

\def\Vu{{V^{\<u\>}}}
\def\Yw{{Y^{\<w\>}}}

\def\eone{\textbf{\ts\textrm{e}$_1$}}
\def\etwo{\textbf{\ts\textrm{e}$_2$}}
\def\bq{\mathbf{q}}

\title[Extensions of the Kahn--Saks inequality]{Extensions of the Kahn--Saks inequality \\ for posets of width two}
\date{\today}

\author{Swee Hong Chan}
\address[Swee Hong Chan]{Department of Mathematics, UCLA,  Los Angeles, CA 90095.}
\email{\texttt{sweehong@math.ucla.edu}}

\author[\ts Igor Pak]{Igor Pak}
\address[Igor Pak]{Department of Mathematics, UCLA,  Los Angeles, CA 90095.}
\email{\texttt{pak@math.ucla.edu}}

\author[\ts Greta Panova]{Greta Panova}
\address[Greta Panova]{Department of Mathematics, USC,  Los Angeles, CA 90089.}
\email{\texttt{gpanova@usc.edu}}

\begin{document}

\begin{abstract}
The Kahn--Saks inequality is a classical result on
the number of linear extensions of finite posets.
We give a new proof of this inequality for posets of width two and both elements in the same chain
using explicit injections of lattice paths.  As a consequence
we obtain a $q$-analogue, a multivariate generalization and
an equality condition in this case.  We also discuss the equality
conditions of the Kahn--Saks inequality for general posets and
prove several implications between conditions conjectured
to be equivalent.
\end{abstract}

\maketitle

\section{Introduction}\label{sec:intro}

\subsection{Foreword}\label{ss:intro-foreword}
The study of linear extensions of finite posets is surprisingly rich as they
generalize permutations, combinations, standard Young tableaux, etc.  By contrast,
the inequalities for the numbers of linear extensions are quite rare and difficult
to prove as they have to hold \emph{for all} posets.  Posets of width two serve
a useful middle ground as on the one hand there are sufficiently many of them
to retain the diversity of posets, and on the other hand they can be analyzed
by direct combinatorial tools.

In this paper, we study two classical results in the area:
the \emph{Stanley inequality} (1981), and its generalization,
the \emph{Kahn--Saks inequality} (1984).
Both inequalities were proved using the geometric \emph{Alexandrov--Fenchel inequalities}
and remain largely mysterious.  Despite much effort, no combinatorial proof of these
inequalities has been found.

We give a new, fully combinatorial, proof of the Kahn--Saks inequality for posets of width two and both elements in the same chain.
In this case, linear extensions are in bijection with certain lattice paths, and we
prove the inequality by explicit injections.  This is the approach first pioneered
in~\cite{CFG,GYY} and more recently extended by the authors in~\cite{CPP2}.
In fact, Chung, Fishburn and Graham~\cite{CFG} proved Stanley's inequality
for width two posets and their conjecture paved a way to Stanley's paper~\cite{Sta}.
The details of our approach are somewhat different, but we do recover the
\emph{Chung--Fishburn--Graham} (CFG) \emph{injection} as a special case.  The construction in this
paper is quite a bit more technical and is heavily based on ideas in our previous
paper~\cite{CPP2}, where we established the \emph{cross-product conjecture}
in the special case of width two posets.

Now, our approach allows us to obtain \emph{$q$-analogues} of both inequalities in the style
of the \emph{$q$-cross-product inequality} in~\cite{CPP2}.  More importantly, it is also
robust enough to imply \emph{conditions for equality} of the Kahn--Saks inequalities for
the case of posets of width two and both elements in the same chain.
The corresponding result for the Stanley inequality
in the generality of all posets was obtained by Shenfeld and van Handel~\cite{SvH}
using technology of geometric inequalities.  Most recently, a completely different
proof was obtained by the first two authors~\cite{CP}.
Although the equality condition in the special case of
the Kahn--Saks inequality is the main result of paper,
we start with a special case of the Stanley inequality as a stepping stone
to our main results.

\smallskip

\subsection{Two main inequalities}\label{ss:intro-main}
Let \ts $P=(X,\prec)$ \ts be a finite poset.
A \defn{linear extension} of $P$ is a bijection \. $L: X \to [n]$, such that
\. $L(x) < L(y)$ \. for all \. $x \prec y$.
Denote by \ts $\Ec(P)$ \ts the set of linear extensions of $P$,
and write \. $e(P):=|\Ec(P)|$.  The following are two key results in the area:

\medskip

\begin{thm}[{\rm \defn{Stanley inequality}~\cite[Thm~3.1]{Sta}}]\label{t:Sta}
Let \ts  $P=(X,\prec)$ \ts be a finite poset, and let \ts $x\in X$.
Denote by \ts $\aNr(k)$ \ts the number of linear extensions \ts $L\in \Ec(P)$,
such that \ts $L(x)=k$.  Then:
\begin{equation}\label{eq:Sta-ineq}
\aNr(k)^2 \,\. \ge \,\. \aNr(k-1) \,\. \aNr(k+1) \quad \text{for all} \quad k\. > \. 1\ts.
\end{equation}
\end{thm}

\medskip

In other words, the distribution of value of linear extensions on~$x$ is \emph{log-concave}.

\medskip

\begin{thm}[{\rm \defn{Kahn--Saks inequality}~\cite[Thm~2.5]{KS}}]\label{t:KS}
Let \ts $x,y\in X$ \ts be distinct elements of a finite poset \ts $P=(X,\prec)$.
Denote by \ts $\aFr(k)$ \ts the number of linear extensions \ts $L\in \Ec(P)$,
such that \ts $L(y)-L(x)=k$.  Then:
\begin{equation}\label{eq:KS-ineq}
\aFr(k)^2 \,\. \ge \,\. \aFr(k-1) \,\. \aFr(k+1) \quad \text{for all} \quad k\. > \. 1\ts.
\end{equation}
\end{thm}

\medskip

Note that the Stanley inequality follows from the Kahn--Saks inequality by
adding the maximal element~$\ts\wh 1$ \ts to the poset~$P$, and
letting \ts $y \gets \wh 1$.

\medskip

\subsection{The $q$-analogues}\label{ss:intro-q-analogues}
From this point on, we consider only posets $P$ of \defn{width two}.  Fix a
partition of~$P$ into two chains \ts $\Cr_1,\Cr_2 \ssu X$,
where \ts $\Cr_1 \cap \Cr_2=\emp$.  Let \ts
$\Cr_1 = \{ \al_1,\ldots, \al_{\ana}\}$ \ts and \ts
$\Cr_2=\{\beta_1,\ldots,\beta_{\bnb}\}$ \ts be these chains
of lengths $\ana$ and $\bnb$, respectively.
The \defn{weight} of a linear extension $L\in \Ec(P)$ is defined in~\cite{CPP2} as
\begin{equation}\label{eq:def-weight}
 \wgt(L) \ := \ \sum_{i=1}^\ana \,  L(\al_i)\..
\end{equation}
Note that the definition of the weight \ts $\wgt(L)$ \ts depends on the chain
partition \ts $(\Cr_1,\Cr_2)$.  We can now state our first two results.

\medskip

\begin{thm}[{\rm \defn{$q$--Stanley inequality}}]\label{t:q-Sta}
Let \ts  $P=(X,\prec)$ \ts be a finite poset of width two, let \ts $x\in X$, and
let \ts $(\Cr_1,\Cr_2)$ \ts be the chain partition as above.
Define
$$
\aNr_q(k) \, := \, \sum_{L\in \Ec(P) \ : \ L(x)=k} \, q^{\wgt(L)}\..
$$
Then:
\begin{equation}\label{eq:q-Sta-ineq}
\aNr_q(k)^2 \ \geqslant \ \aNr_q(k-1) \, \aNr_q(k+1) \quad \ \text{ for all} \ \quad k\.>\.1 \ts,
\end{equation}
where the inequality between polynomials in \. $q$ \. is coefficient-wise.
\end{thm}

\smallskip

The following result is a generalization, sice we can always assume that element
\ts $y = \wh 1$ \ts is in the same chain as element~$x$.

\smallskip

\begin{thm}[{\rm \defn{$q$--Kahn--Saks inequality}}]\label{t:q-KS}
Let \ts $x,y\in X$ \ts be distinct elements of a finite poset \ts $P=(X,\prec)$ \ts
of width two.
Suppose that either \. $x,y \in \Cr_1$\ts, or \. $x,y \in \Cr_2$.
  Define:
$$
\aFr_q(k) \ := \ \sum_{L\in \Ec(P) \ : \ L(y)-L(x)=k} \, q^{\wgt(L)}\..
$$
Then:
\begin{equation}\label{eq:q-KS-ineq}
\aFr_q(k)^2 \, \geqslant \, \aFr_q(k-1) \. \aFr_q(k+1) \quad \ \text{for all} \ \quad k\. > \. 1\ts,
\end{equation}
where the inequality between polynomials in \ts $q$ \ts is coefficient-wise.
\end{thm}

\smallskip

In Section~\ref{sec:multi}, we give a multivariate generalization of both theorems.
Note that the assumption that $x$ and $y$ belong to the same chain in the partition
\ts $(\Cr_1,\Cr_2)$ \ts  are necessary for the conclusion of  Theorem~\ref{t:q-KS} to hold, as shown in the next example.

\smallskip

\begin{ex}\label{ex:Ramon}
Let \. $P= C_3 + C_3$ \. be the disjoint sum of two chains with three elements.
Denote these chains by  \ts $\Cr_1:=\{\al_1,\al_2,\al_3\}$ \ts
  and \ts $\Cr_2:=\{\be_1,\be_2,\be_3\}$. For elements \. $x= \al_1$ \. and \. $y=\be_3$,
we have:
$$\aF_q(1) \, = \,  q^{14}\., \quad  \aF_q(2) \, = \, 2 \ts q^{13}  \quad \text{and} \quad \aF_q(3)\, = \, 3 \ts q^{12} \. + \. q^{11}\ts.
$$
We conclude:
\begin{align*}
\aF_q(2)^2  \, - \,  \aF_q(1) \, \aF_q(3)
 \ = \  q^{26} \. - \. q^{25}   \ \not \geqslant  \ 0\ts.\end{align*}
\end{ex}

\medskip

\subsection{Equality conditions}\label{ss:intro-equality}
Let $x= \al_r\in \Cr_1$.  We say that $x$ satisfies a \ts \defn{$k$-pentagon property} \ts if
$$
\al_{r-1} \, \prec \, \beta_{k-r} \, \prec \, \beta_{k-r+1} \, \prec \, \al_{r+1} \qquad
\text{and} \qquad  \al_r \.||\. \beta_{k-r} \ \,, \ \  \al_r \.||\. \beta_{k-r+1}\ \,,
$$
where \ts $u\.||\. v$ \ts denotes \emph{incomparable elements} \ts $u,v\in X$.
In other words, the  subposet of~$P$ restricted to
$$
\{\al_{r-1}, \al_{r}, \al_{r+1}, \beta_{k-r}, \beta_{k-r+1}\}
$$
has a pentagonal Hasse diagram, see Figure~\ref{fig:pentagon}.
For  \ts $x=\beta_r\in \Cr_2$ \ts the $k$-pentagon property is defined analogously.

\begin{figure}[hbt]
\begin{center}
\includegraphics[width=3.2cm]{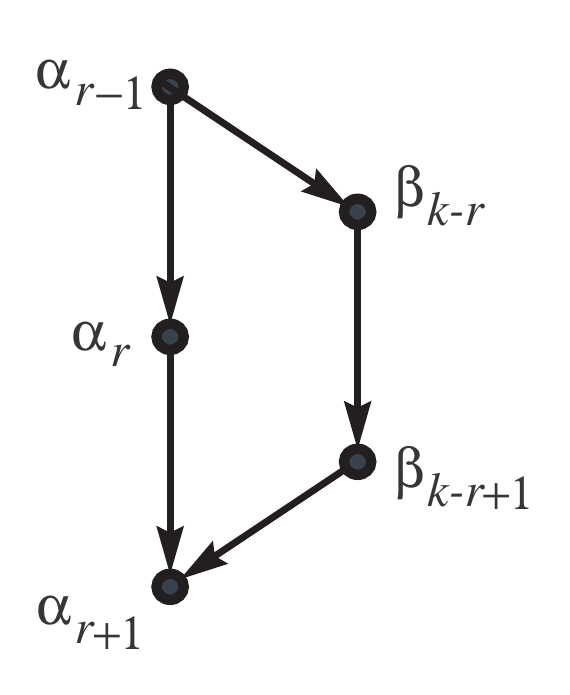}
\end{center}
\caption{The $k$-pentagon property for \ts $x=\al_r\in \Cr_1$.
The arrows point from smaller to larger elements in the poset.}\label{fig:pentagon}
\end{figure}

\medskip

\begin{thm}[{\rm \defn{Equality condition for the $q$-Stanley inequality},
cf.~Theorem~\ref{t:Sta-equality-gen}}]\label{t:q-Sta-equality}
Let \ts  $P=(X,\prec)$ \ts be a finite poset of width two.  Fix \ts $x\in X$,
and let \ts $\aNr(k)$, $\aNr_q(k)$ \ts be defined as above.
Suppose that \ts $k \in \{1,\ldots,n-1\}$ \ts and  \ts $\aNr(k)>0$.  Then the following are equivalent:
\begin{enumerate}
			[{label=\textnormal{({\alph*})},
		ref=\textnormal{\alph*}}]
\item \label{itemequality Stanley a} \ $\aNr(k)^2 \, = \, \aNr(k-1) \, \aNr(k+1)$,
\item \label{itemequality Stanley b} \ $\aNr(k) \, = \, \aNr(k+1) \. = \. \aNr(k-1)$,
\item \label{itemequality Stanley c} \ $\aNr_q(k)^2 \, = \, \aNr_q(k-1) \, \aNr_q(k+1)$,
\item \label{itemequality Stanley d} \ $\aNr_q(k)\, = \,  q^\ve \ts\aNr_q(k-1)\, = \,  q^{-\ve}\ts \aNr_q(k+1)$, \.
where \ts $\ve=1$ \ts for \ts $x \in \Cr_1$ \ts  and \ts $\ve=-1$ \ts for \ts  $x \in \Cr_2$\ts,
\item  \label{itemequality Stanley e} \ element $x$ satisfies \ts $k$-pentagon property.
\end{enumerate}
\end{thm}

\smallskip

The equivalence \. (a)~$\Leftrightarrow$~(b) was recently proved by Shenfeld
and van Handel~\cite{SvH} for general posets via a condition implying~(e),
see Theorem~\ref{t:Sta-equality-gen} and the discussion that follows.
Conditions~(c) and~(d) are specific to posets of width two.
The following result is a generalization of Theorem~\ref{t:q-Sta-equality} and
the main result of the paper:

\medskip

\begin{thm}[{\rm \defn{Equality condition for the $q$-Kahn--Saks inequality}}]\label{t:q-KS-equality}
Let \ts $x,y\in X$ \ts be distinct elements of a finite poset \ts $P=(X,\prec)$ \ts of width two.
Let \ts $\aFr(k)$, $\aFr_q(k)$ \ts be defined as above.
Suppose that either $x,y \in \Cr_1$ or $x,y \in \Cr_2$.
Also suppose that \ts $k \in \{2,\ldots,n-2\}$ \ts and
 \ts  $\aFr(k)>0$. Then the following are equivalent:
\begin{enumerate}
		[{label=\textnormal{({\alph*})},
		ref=\textnormal{\alph*}}]
\item \label{itemequality KS}
\ $\aFr(k)^2 \, = \, \aFr(k-1)\,\aFr(k+1)$,
\item \label{itemequality linear}
\ $\aFr(k) \, = \,\aFr(k+1) \, = \, \aFr(k-1)$,
\item \label{itemequality q-KS}
\ $\aFr_q(k)^2 \, = \, \aFr_q(k-1)\, \aFr_q(k+1)$,
\item \label{itemequality q-linear}
\ $\aFr_q(k) \, = q^{\ve}\ts \aFr_q(k-1)\, = \, q^{-\ve}\ts\aFr_q(k+1)$, for some \. $\ve \in \{\pm 1\}$,
\item \label{itemequality combinatorial} there is an element \ts $z\in \{x,y\}$, such that
for every \ts $L\in \Ec(P)$ \ts for which $L(y)-L(x)=k$, \\
there are elements \ts $u,v \in X$ \ts which satisfy \. $u\. || \.z$, \. $v\. || \. z$, and \.
$L(u)+1=L(z) = L(v)-1$.
\end{enumerate}
\end{thm}

\smallskip

Note that conditions~(c) and~(d) are specific to posets of width two.
While conditions~(a) and~(b) do extend to general posets,
the equivalence  \eqref{itemequality KS} \. $\Leftrightarrow$ \. \eqref{itemequality linear}
does not hold in full generality.   Even for the poset \. $P=C_3+C_3$ \. of width two \.
given in Example~\ref{ex:Ramon}, we have \. $\aF(2)^2 \. = \. \aF(3) \. \aF(1) \. =\. 4$, \. even though \. $\aF(1)=1$,
\. $\aF(2)=2$ \. and \. $\aF(3)=4$.

We should also mention that the \ts  $\aF(k)>0$ \ts assumption is a very weak constraint,
as the vanishing can be completely characterized for general posets (see Theorem~\ref{t:F-positive}).
We refer to Section~\ref{sec:posets} for further discussion of general posets, and for the \emph{$k$-midway
 property} which generalizes the $k$-pentagon property but is more involved.

\smallskip

\subsection{Proof discussion}\label{ss:intro-proof}  As we mentioned above, we start by
translating the problem into a natural question about directed lattice paths in a row/column
convex region in the grid (cf.~$\S$\ref{ss:finrem-region}).
From this point on, we do not work with posets and the proof
becomes purely combinatorial enumeration of lattice paths.

While the geometric proofs in~\cite{KS,Sta} are quite powerful, the equality cases
of the Alexandrov--Fenchel inequality are yet to be fully understood. So proving the
equality conditions of poset inequalities is quite challenging, see~\cite{SvH,CP}
and~$\S$\ref{ss:finrem-hist}.  This is why our direct combinatorial approach is so useful,
as the explicit injection becomes a bijection in the case of equality.

In the case of Stanley's inequality the CFG injection is quite simple and elegant,
leading to a quick proof of the equality condition.  For the Kahn--Saks inequality,
the direct injection is a large composition of smaller injections, each of which
is simple and either generalizes the CFG injection or of a different flavor, all
influenced by the noncrossing paths in the \emph{Lindstr\"om–-Gessel–-Viennot}
\emph{lemma}~\cite{GV} (see also~\cite[$\S$5.4]{GJ}).
Consequently, the equality condition of the Kahn--Saks
inequality is substantially harder to obtain as one has to put together the
equalities for each component of the proof and do a careful case analysis.

In summary, our proof of the main result (Theorem~\ref{t:q-KS-equality}) is like
an elaborate but delicious dish: the individual ingredients are elegant and natural,
but the instruction on how they are put together is so involved the resulting
recipe may seem difficult and unapproachable.

\smallskip

\subsection{Structure of the paper}\label{ss:intro-structure}
We start with an introductory Section~\ref{sec:basic}
 on posets, lattice paths, and lattice path
inequalities.  This section also includes some reformulated key lemmas from our previous
paper~\cite{CPP2}, whose proof is sketched both for clarity and completeness.
A reader very familiar with the standard definitions, notation and the results
in~\cite{CPP2} can safely skip this section.

In the next Section~\ref{sec:toolkit}, we introduce key combinatorial lemmas which
we employ throughout the paper: a
\emph{criss--cross inequality} (Lemma~\ref{l:path-averages}),
and two \emph{equality lemmas} (Lemma~\ref{l:bijection 1 equality}
and Lemma~\ref{l: lgv}).
In a short Section~\ref{sec:Sta}, we prove both the Stanley inequality (Theorem~\ref{t:Sta})
which easily extends to the proof of the $q$-Stanley inequality (Theorem~\ref{t:q-Sta}),
and the equality conditions for Stanley's inequality (Theorem~\ref{t:q-Sta-equality}).
Even though these results are known in greater generality (except for Theorem~\ref{t:q-Sta}
which is new), we recommend the reader not skip
this section, as the proofs we present use the same approach as the following sections.

In Sections~\ref{s:q-KS-proof} and~\ref{s:q-KS-equality}, we present the proofs
of Theorems~\ref{t:q-KS} and~\ref{t:q-KS-equality}, respectively,
by combining the previous tools together.
These are the central sections of the paper.
In a short Section~\ref{sec:multi},
we give a multivariate generalizations of our $q$-analogues.  Finally, in
Section~\ref{sec:posets}, we discuss generalizations of Theorem~\ref{t:q-KS-equality}
to all finite posets.  We state Conjecture~\ref{conj:KS-equality} characterizing the
\emph{complete equality conditions}i and prove several implications in support of the conjecture
using the properties of promotion-like maps (see~$\S$\ref{ss:finrem-promo}).
We conclude with final remarks and open problems in Section~\ref{sec:finrem}.

\bigskip

\section{Lattice path inequalities} \label{sec:basic}

\subsection{Basic notation} \label{ss:back-not}
We use \ts $[n]=\{1,\ldots,n\}$, \ts $\nn = \{0,1,2,\ldots\}$, and
\ts $\pp = \{1,2,\ldots\}$. Throughout the paper we use $q$ as a variable.
For polynomials \ts $f, g\in \zz[q]$, we write \ts $f\leqslant g$ \ts if
the difference \ts $(g-f) \in \nn[q]$, i.e.\ if \ts $(g-f)$ \ts
is a polynomial with nonnegative coefficients.  Note the difference
between relations
$$
x \ts \preccurlyeq y\,, \quad a \ts\le \ts b \,  \quad  \text{and} \quad f \leqslant \ts g \ts,
$$
for posets elements, integers and polynomials, respectively.

\smallskip

\subsection{Lattice path interpretation}\label{ss:lattice path interpretation}

Let $P=(X,\prec)$ be a finite poset of width two and let \ts $(\Cr_1,\Cr_2)$ \ts be a fixed
partition into two chains.  Denote by \ts $\zero=(0,0)$ \ts the origin and by
\ts $\eone = (1,0)$, \ts
$\etwo=(0,1)$ \ts two standard unit vectors in $\Zb^2$.

For a linear extension $L\in \Ec(P)$, define the \defn{North--East {\rm (NE)} lattice path}
\ts $\phi(L)$ obtained from $L$ by interpreting it as a sequence of North and East
steps corresponding to elements in $\Cr_1$ and $\Cr_2$, respectively.
Formally, let \ts
$\phi(L):=(Z_t)_{1 \leq t \leq n}$ \. in \ts $\Zb^2$ \ts
from \ts $\zero=(0,0)$ to $(\ana,\bnb)$, be the path \ts
defined recursively as follows:
\begin{equation*}
	Z_0 \, = \, \zero, \qquad  Z_t \ := \
	\begin{cases}
		\ts Z_{t-1} \ts + \ts \eone & \text{ if } \ L^{-1}(t) \in \Cr_1\.,\\
		\ts Z_{t-1}\ts + \ts \etwo & \text{ if } \ L^{-1}(t) \in \Cr_2\..
	\end{cases}
\end{equation*}
%

\nin
Denote by $\Cen(P)$ the set
\begin{align*}
 \Cenup(P) \ &:= \ \bigg\{ \left(h-\frac{1}{2}, k - \frac{1}{2}\right) \in \Rb^2 \ :  \  \alpha_h \. \precc \. \beta_k\,, \ 1\le h \le \ana, \ 1\le k \le \bnb \. \bigg \}\., \\
  \Cendown(P) \ &:= \ \bigg\{ \left(h-\frac{1}{2}, k - \frac{1}{2}\right) \in \Rb^2 \ :  \  \alpha_h \. \succc \. \beta_k\,, \ 1\le h \le \ana, \ 1\le k \le \bnb \. \bigg \}\..
\end{align*}
Let $\Forbup(P)$ and $\Forbdown(P)$ be  the set of unit squares in $[0, \ana] \times [0,\bnb]$
whose centers are in $\Cenup(P)$ and $\Cendown(P)$, respectively.
Note that the region $\Forbup(P)$ lies above the region $\Forbdown(P)$, and their interiors do not intersect.
Let $\Reg(P)$ be the (closed) region of $[0,\ana]\times [0,\bnb]$ that
is bounded from above by the region $\Forbup(P)$, and from below by the region $\Forbdown(P)$,
see Figure~\ref{f:regionP}.
It follows directly from the definition that $\Reg(P)$ is a connected row and
column convex region, with boundary defined by two lattice paths. Moreover, the lower boundary of $\Reg(P)$ is the lattice path corresponding to the $\Cr_1$-minimal linear extension (i.e. assigning the smallest possible values to the elements of $\Cr_1$), and the upper boundary corresponds to the $\Cr_1$-maximal linear extension.

\begin{figure}[hbt]
	\centering
	\begin{tabular}{c@{\hskip 1 in}c }
		\includegraphics[width=0.18\linewidth]{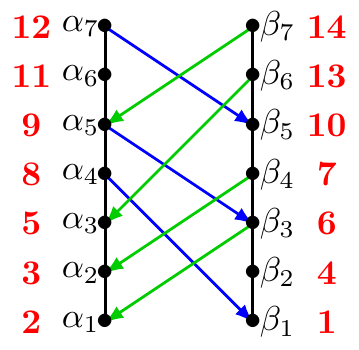}  &
		\includegraphics[width=0.2\linewidth]{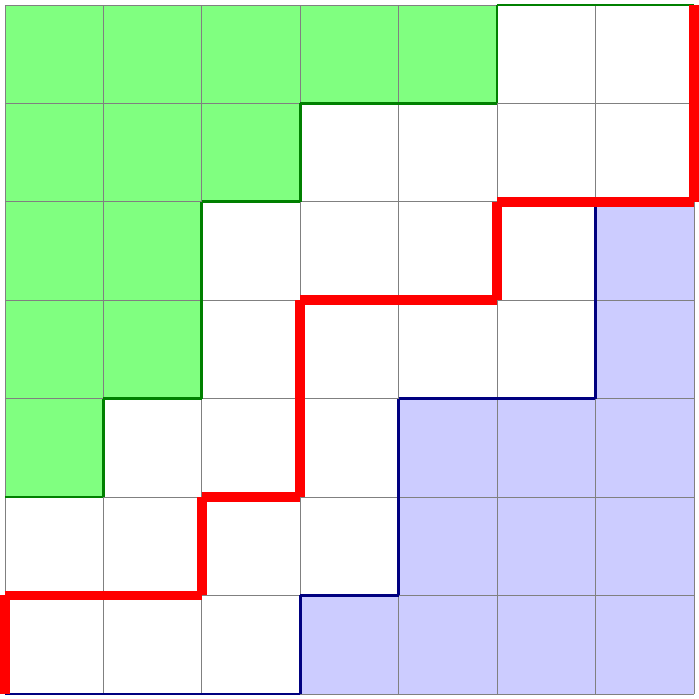}  \\
		(a) & (b)
	\end{tabular}
	\caption{(a) The Hasse diagram of a poset $P$, and a linear extension $L$ of $P$ (written in red).
	(b) The corresponding region $\Reg(P)$, with $\Forbup(P)$ in green and $\Forbdown(P)$ in blue, and the lattice path $\phi(L)$ in red.	
	}
	\label{f:regionP}
\end{figure}

\smallskip

\begin{lemma}[{\rm \cite[$\S$2]{CFG} and~\cite[Lem~8.1]{CPP2}}]\label{l:interpretation lattice path}
	The map~$\phi$ described above is a bijection between
    \ts $\Ec(P)$ \ts and NE lattice paths in \ts $\Reg(P)$ \ts
    from \ts $\zero$ \ts to \ts $(\ana,\bnb)$.
\end{lemma}

\smallskip

\begin{rem}\label{rem:general-regions}
{\rm It is not hard to see the regions \ts $\Reg(P)$ \ts which appear in
Lemma~\ref{l:interpretation lattice path} have no other constraints.
Formally, for every region \ts $\Gamma\ssu \zz^2$ \ts between two
noncrossing paths \. $\gamma,\gamma': \ts \zero \to (\ana,\bnb)$, there
is a poset $P$ of width two with a partition into two chains of sizes \ts
$\ana$ and~$\bnb$, such that \ts $\Gamma = \Reg(P)$. We leave the proof
to the reader, see also~$\S$\ref{ss:finrem-region}.
}\end{rem}

\medskip

\subsection{Inequalities for pairs of paths}\label{ss:lemmas-ineq}

We will use the lattice path inequalities from~\cite{CPP2} and prove their extensions. In order to explain the combinatorics, we will briefly describe the proofs from~\cite{CPP2}. Informally, they state that there are more pairs of paths which pass closer to the \emph{inside} of the region than to the \emph{outside} of the region.

Let \ts $A, B \in \Reg(P)$.
Denote by \. $\Kc(A,B)$  \. the set of NE lattice paths \ts $\zeta: \A \to \B$, such that \ts
    $\zeta\in \Reg(P)$. Similarly,
    denote by \. $\aK_q(\A,\B)$  \. the polynomial
	\begin{align*}
		\aK_q(\A,\B) \  := \  \sum_{\zeta \in \Kc(\A,\B)} \. q^{\wgt(\zeta)}\,,
	\end{align*}
	and we write \. $\aK(A,B) := \aK_1(A,B)$ \. (i.e., when $q=1$).
\medskip

\begin{lemma}[{\rm \cite[Lem~8.2]{CPP2}}]
\label{l:lattice path bijection 1}
Let \ts $A, A', B', B \in \Reg(P)$ \ts be on the same vertical line with $A$ above $A'$ such that $\oa{AA'} = -\oa{BB'}$ and $A'$ on or above $B$, i.e.\ \ts $a_1=a_1'=b_1=b_1'$ \ts and
\ts $a_2 - a_2' =  b_2' - b_2$ with $a_2' \geq b_2$.
Let \ts $C,D \in \Reg(P)$ \ts be on a vertical line to the right
of the line $AB$, and such that \ts $a'_2-b_2 \geq c_2-d_2$.
Then:
	\begin{equation*}
				\aKr_q(\A',\C) \,\cdot \, \aKr_q(\B',\D) \  \geqslant \  	\aKr_q(\A,\C) \,\cdot \, \aKr_q(\B,\D).
			\end{equation*}			
	\end{lemma}
\begin{figure}[hbt]
\begin{center}
\includegraphics[width=0.9\textwidth]{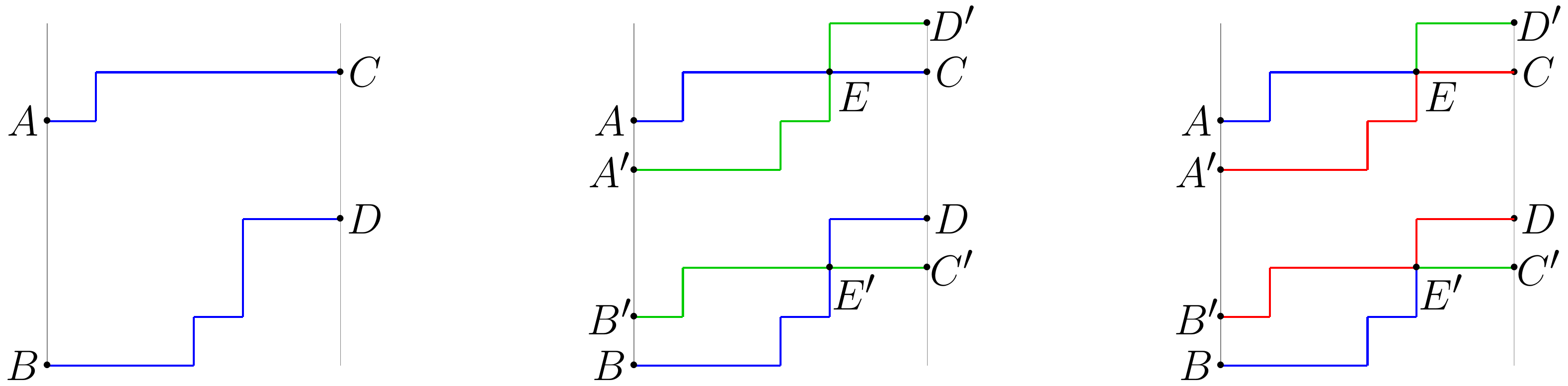}
\end{center}
\caption{The proof of Lemma~\ref{l:lattice path bijection 1}: The injection $\vk$ takes the blue paths $A \to C, B \to D$, translates the $B \to D$ path up to form the green path $A' \to D'$ (second picture), intersects it with the blue $A \to C$ path at $E$, and then forms the red path $A' \to C$ (by following the green $A' \to E$ and then switching to the blue $E \to C$. The other red path is obtained by translating the blue/green $A \to E \to D'$ down.  }\label{fig:injection1}
\end{figure}

\begin{proof}[Proof outline]
We exhibit an injection $\vk$ from pairs of paths
\ts $\gamma: A \to C$, \ts $\delta:B \to D$ \ts in \ts $\Reg(P)$ \ts
to pairs of paths \ts $\gamma':A' \to C$, \ts $\delta': B' \to D$ \ts in \ts $\Reg(P)$.
Let \ts $\vb = \oa{BA'}$ \ts and \ts $\wh{\delta} = \delta +\vb$ \ts be
the translated path~$\delta$, which starts at \ts $A' = B+ \vb$ \ts and
ends at \ts $D' = D+ \vb$, lying on or above~$C$ by the condition in the Lemma.
Then $\gamma$ and $\wh{\delta}$ must intersect, and let $E$ be their first (closest to $A$) intersection point.

Now, let \. $\gamma' = \wh{\delta}(A',E) \circ \gamma(E,C)$, so \ts $\gamma':A' \to C$.
Similarly, let \. $\delta'= \gamma(A,E) \circ \wh{\delta}(E,D') \ts - \ts\vb$, so \ts $\delta': B' \to D$. Then \. $\gamma' \subset \Reg(P)$ \. since $\wh{\delta}$ is on or above \ts $\delta \subset \Reg(P)$ (because $a_2 \geq b_2$) and is strictly below $\gamma\subset \Reg(P)$ since $E$ is the first intersection point. Similarly, $\gamma(A,E) - \vb$ \ts is also between $\gamma$ and $\delta$ and hence in $\Reg(P)$. The other parts of \ts $\gamma', \delta'$ \ts are part of the original paths $\gamma,\delta$ and so are also in~$\Reg(P)$. Then $\vk$ is clearly an injection.
Since the paths are composed of the same pieces, some of which translated vertically with zero net effect, the total $q$-weight is preserved.
\end{proof}

\bigskip

\section{Lattice paths toolkit expansion}\label{sec:toolkit}

\subsection{Criss-cross inequalities}\label{sec:toolkit-criss-cross}

Here we consider inequalities between sums of pairs of paths.

\smallskip

\begin{lemma}[{\rm \defn{Criss-cross lemma}}]
\label{l:path-averages}
Let \. $A,A',B',B \in \Reg(P)$ \. be on the same vertical line, with $A$ the highest and $B$ the lowest points.
In addition, let \. $C,C',D,D'\in \Reg(P)$ be on another vertical line, with $C$ the highest and $D$ the lowest
points, and such that \. $\oa{CC'}=-\oa{DD'}=\oa{AA'}=-\oa{BB'}$.
Finally, let \. $\oa{AB} = \oa{CD}$. Then we have:
\begin{equation}\label{eq:double_path_diff}
\begin{split}
				&\aKr_q(A,C) \,\cdot \, \aKr_q(B,D) \ + \  \aKr_q(A',C') \,\cdot \, \aKr_q(B',D') \\
& \qquad \geqslant \   	\aKr_q(A',C) \,\cdot \, \aKr_q(B',D) \ + \ \aKr_q(A,C') \,\cdot \, \aKr_q(B,D') .
				\end{split}
			\end{equation}	
\end{lemma}

\begin{figure}[h!]
\includegraphics[width=1.5in]{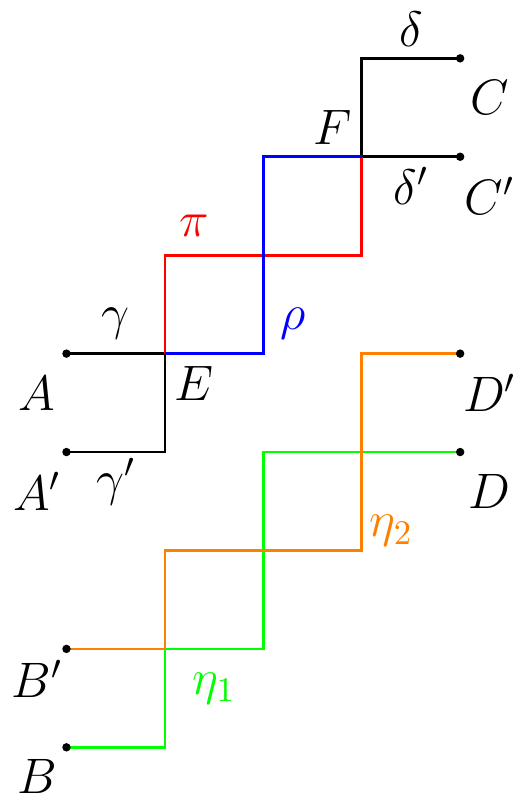}
\caption{ Illustration of the proof of Lemma~\ref{l:path-averages}. Here we show points \ts $E,F$ \ts and paths \ts $\gamma, \gamma', \delta,\delta'$, $\pi$ (in red) \ts and~$\rho$ (in blue). The green path $\eta_1$ is formed by following $\gamma'$ then $\rho$ and then $\delta'$ and translating the resulting path by $\vb$, so it is a path $B \to D$. The orange path $\eta_2$ is also shown. } \label{fig:path_averages1}
\end{figure}

\smallskip

\begin{proof}
The idea is to consider the pairs of paths counted on each side, and show that each pair (after the necessary transformation) is counted less times on the RHS than on the LHS, where the number of times it could appear on each side is $0,1,2$.

To be precise, given two points $E$ and $F$ in $\Reg(P)$ between the lines $AB$ and $CD$, and paths $(\pi,\rho)$  with endpoints $E$ and $F$, let
$$S(E,F) \ := \ \big\{ (\ga, \. \ga', \. \delta, \. \delta') \ | \ \ga:A\to E, \, \ga':A'\to E, \, \delta:F \to C, \, \delta': F \to C' \big\}.
$$
Here we have $4$-tuples of paths with the given endpoints, such that their only intersection points are the endpoints, namely \. $\ga \cap \ga' =\{E\}$ \. and \, $\delta \cap \delta'=\{F\}$. Connecting  the paths in $S(E,F)$  with $(\pi,\rho)$, we can obtain four different pairs of paths from the points $A,A'$ to $C,C'$. We now count how often each such pair is counted in LHS and RHS of the desired inequality in~\eqref{eq:double_path_diff}, after we translate one of the paths by \. $\vb := \oa{AB'} = \oa{A'B} = \oa{C'D} = \oa{CD'}$.

Fix points $E,F$ as above, paths $\pi,\rho:E \to F$, and 4-tuple $(\ga,\ga',\delta,\delta') \in S(E,F)$. These 6 paths can be combined in different ways to give 2 paths from $A,A'$ to $C,C'$, and after translating one by $\vb$ obtain pairs appearing in~\eqref{eq:double_path_diff}.
The pairs are:
\begin{align*}
&\zeta_1 \ := \ \ga \circ \pi \circ \delta, \qquad \zeta_1: A \to C\ts,   \quad &\eta_1 \ := \ (\ga' \circ \rho \circ \delta') +\vb, \qquad \eta_1: B \to D\ts,  \\
&\zeta_2 \ := \ \ga' \circ \pi \circ \delta', \qquad \zeta_2:A' \to C'\ts,  \quad &\eta_2 \ :=  \ (\ga \circ \rho \circ \delta) +\vb, \qquad \eta_2: B' \to D'\ts, \\
&\zeta_3 \ := \ \ga \circ \pi \circ \delta', \qquad \zeta_3: A \to C'\ts,   \quad &\eta_3 \ := \ (\ga' \circ \rho \circ \delta) +\vb, \qquad \eta_3: B \to D'\ts, \\
&\zeta_4 \ := \ \ga' \circ \pi \circ \delta, \qquad \zeta_4:A' \to C\ts,  \quad  &\eta_4\ := \ (\ga \circ \rho \circ \delta') +\vb, \qquad\eta_4: B' \to D\ts.
\end{align*}

\smallskip

\nin
{\bf Case 1:} \ts  At least one of $\zeta_3, \eta_3$ is not (entirely contained) in $\Reg(P)$, and at least one of $\zeta_4, \eta_4$ is not  in $\Reg(P)$, then none of these pairs of paths is counted in the RHS of~\eqref{eq:double_path_diff}, and the contribution to the RHS is 0.

\smallskip

\nin
{\bf Case 2:} \ts Both pairs of paths  $(\zeta_3,\eta_4)$ and $(\zeta_4,\eta_4)$ are contained in $\Reg(P)$. This implies that all the components and their translates are in $\Reg(P)$, and hence \. $\zeta_1,\zeta_2,\eta_1,\eta_2 \subset \Reg(P)$. So the contribution from these paths is 2 on both LHS and RHS.

\smallskip

\nin {\bf Case 3 and 4:} \ts Exactly one pair is in $\Reg(P)$, say $\zeta_3,\eta_3 \subset \Reg(P)$ and at least one of $\zeta_4,\eta_4$ is not  in $\Reg(P)$. Then $\ga, \delta', \ga' +\vb, \delta +\vb \subset \Reg(P)$. Since $\ga'$ is between $\ga$ and $\ga'+\vb$, both of which are contained in~$\Reg(P)$, and since $\Reg(P)$  is simply connected, we conclude that  $\ga'$ is also in $\Reg(P)$. Thus, $\zeta_2 \subset \Reg(P)$. Similarly, since $\ga +\vb$ is between $\ga$ and $\ga'+\vb$, we have $\ga+\vb \subset \Reg(P)$. Thus, $\zeta_2, \eta_2 \subset \Reg(P)$. Hence these paths are counted once in the RHS and at least once in the LHS.

To finish the proof, we need to show that we have indeed considered all possible pairs of paths which can arise in the RHS. Let $\zeta \in \Kc(A',C)$, $\eta \in \Kc(B',D)$, so $(\eta,\zeta)$ is a pair of paths counted in the first term on the RHS. Let $\wh{\eta} = \eta-\vb:A \to C'$, it has to intersect $\zeta$. Let $E$ be the first intersection point (closest to $A/A'$) and let $F$ be the last intersection point. Set $\pi = \zeta(E,F)$, $\rho = \wh{\eta}(E,F)$ and $\ga' = \zeta(A',E)$, $\ga = \wh{\eta}(A,E)$, $\delta'=\wh{\eta}(F,C')$ and $\delta=\zeta(F,C)$. Then, fixing these $E,F, \pi, \rho$ and $(\ga,\ga',\delta,\delta') \in S(E,F)$ we recover $\zeta = \zeta_4$ and $\eta = \eta_4$. Similarly, given $\zeta \in \aKr(A,C')$ and $\eta \in \aKr(B,D')$ we recover $(\zeta_3,\eta_3)$.

Moreover, these constructions reassign portions of the same paths on the RHS and LHS, total translated areas cancel out, so the $q$-weights are preserved and the inequality holds for the $q$-weighted paths. This completes the proof.
\end{proof}

\medskip

\subsection{Equalities}

Here we describe the cases when equalities in the lattice path lemmas from Section~\ref{sec:basic} are achieved.
The following is an easy generalization of the~\cite[Lemma 8.4]{CPP2}.

\begin{lemma}[{\rm \defn{Equality lemma}}]\label{l:bijection 1 equality}
	Let \ts $A,B,A',B', C, D \in \Reg(P)$ \ts be as in Lemma~\ref{l:lattice path bijection 1}.
We then have the following conditions for equalities in Lemma~\ref{l:lattice path bijection 1}: \
	If \. $a_2'-b_2 > c_2-d_2$\ts, then
	\begin{equation*}
				\aKr(A',C) \,\cdot \, \aKr(B',D) \  = \  	\aKr(A,C) \,\cdot \, \aKr(B,D)
			\end{equation*}
			if and only if either both sides are zero, or
				\begin{equation*}
				\aKr_q(A',C) \ = \ \aKr_q(A,C) \ \quad \text{and} \ \quad   \aKr_q(B',D) \  = \  	 \aKr_q(B,D) .
			\end{equation*}
			Furthermore, if \ts  $a_2>a_2'$ \ts and the segment $CD$ lies strictly to the right of segment $AB$,
			then the segment $AB$ is part of the lower boundary of $\Reg(P)$.
\end{lemma}

\smallskip

\begin{proof}
We assume that $a_2>a_2'$ and the segment $CD$ lies strictly to the right of $AB$, as otherwise the lemma is straightforward.
The equality in Lemma~\ref{l:lattice path bijection 1} implies that the map $\vk$ is a bijection. Let \. $\xi:B' \to D$ \. be the highest possible path in $\Reg(P)$ and \. $\eta:A' \to C$ \. be the lowest possible path in $\Reg(P)$, see~Figure~\ref{fig:equality}. Then these paths must be in the image of $\vk$, and their preimages are \. $\wh{\xi}:B \to D$ \. and $\wh{\eta}:A \to C$. Let \. $\vb = \oa{B'A}$.

Following the construction of $\vk^{-1}$, we see that the paths $\eta$ and \ts $\xi + \vb$ \ts must intersect, with $E$ the closest intersection point to $A$. By the minimality of $\eta$ and maximality of $\xi$ in $\Reg(P)$, we have that \ts $\xi +\vb$ \ts is on or above $\eta$. Since
the endpoints of $\xi+\vb$ (i.e. $A$ and~$D'$) are strictly above the endpoints of $\eta$ (i.e. $A'$ and~$C$) by assumptions,
we have  $E$ is contained in lower boundary of~$\Reg(P)$.
Since $\xi$ is below \ts $\xi+\vb$ and is above the lower boundary of \ts $\Reg(P)$,
we have $E$ is contained in $\xi$. Next, we observe that if \ts $E \not \in AB$, then \ts $\wh{\eta}(A,E)$ \ts is strictly above $\xi(B',E)$, which contradicts the maximality of $\xi$ in~$\Reg(P)$. Thus $E$ is contained in $AB$ and is on or above~$A$,
and so the lower boundary of $\Reg(P)$ contains the segment~$AB$.  This completes the proof.
\end{proof}

\begin{figure}
\includegraphics[width=2in]{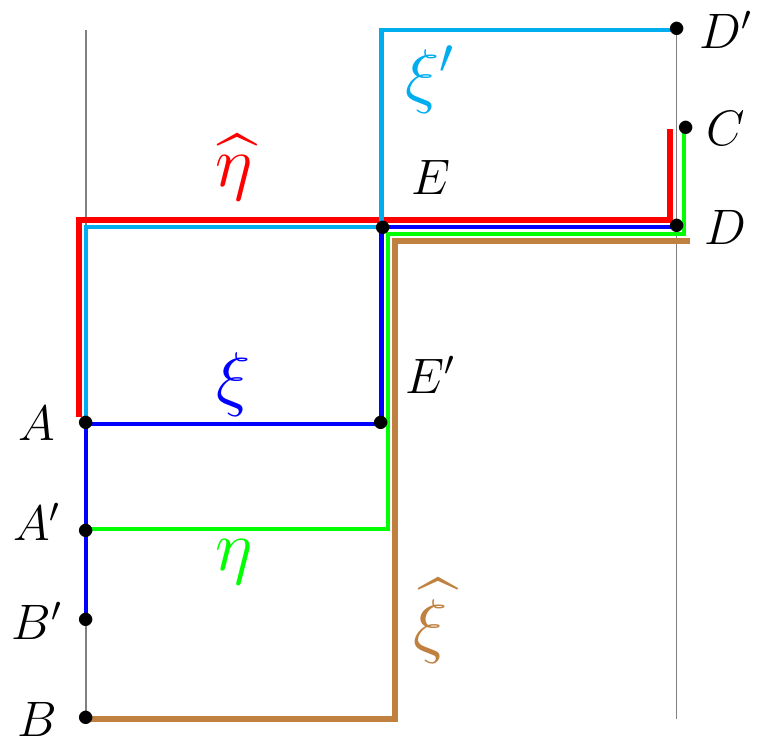}
\caption{The proof of Lemma~\ref{l:bijection 1 equality}.}\label{fig:equality}
\end{figure}

\medskip

The following Lemma treats the special case when $A'=B$ in the Equality Lemma~\ref{l:bijection 1 equality}.
The inequality itself reduces directly to Lindstr\"om--Gessel--Viennot lemma
as the translation vector \. $\vb=\zero$.

\smallskip

\begin{lemma}[{\rm \defn{Special equality lemma}}]\label{l: lgv}
	Let  \.$A,B \in \Reg(P)$ \. be two points on the same vertical line with $A$ above~$B$,
and \. $C,D \in \Reg(P)$ \. points on another vertical line with $C$ above~$D$  to the east
of the line~$AB$. Then:
$$
\aKr_q(A,C)\, \cdot\, \aKr_q(B,D) \ \geqslant \ \aKr_q(B,C) \, \cdot\, \aKr_q(A,D)
$$
with equality if and only if there exists a point $E$ for which every path counted here must pass through, i.e.,
	\begin{alignat*}{2}
		&\aKr_q(A,C) \ = \ \aKr_q(A,E) \. \cdot \. \aKr_q(E,C),  \qquad &&\aKr_q(B,D) \ = \ \aKr_q(B,E) \. \cdot \.  \aKr_q(E,D),  \\
		&\aKr_q(B,C) \ = \ \aKr_q(B,E)  \. \cdot \. \aKr_q(E,C),  \qquad &&\aKr_q(A,D) \ = \ \aKr_q(A,E)  \. \cdot \. \aKr_q(E,D).
	\end{alignat*}
Furthermore, if $CD$ lies strictly to the right of $AB$, then  one of the three conditions hold:
	\begin{itemize}
		\item[{\rm (a)}] \. $E=A$ \. is part the lower boundary of \ts $\Reg(P)$,
		\item[{\rm (b)}] \. $E= D$ \. is part of the upper boundary of \ts $\Reg(P)$,
		\item[{\rm (c)}] \. $E$ is part of the upper and lower boundary of \ts $\Reg(P)$.
	\end{itemize}
\end{lemma}

\begin{proof}
We assume that segment $CD$ lies strictly to the right of $AB$, as otherwise the lemma is straightforward.
First, observe that the inequality follows from Lemma~\ref{l:lattice path bijection 1} by setting \. $A' \leftarrow A$, \. $B' \leftarrow B$ \. and \. $A \leftarrow B$, \. $B \leftarrow A$. In that case the translation vector is zero and we apply the intersection argument directly to the paths \ts $A \to C, B \to D$.
	
To analyze the equality, we notice that Lemma~\ref{l:bijection 1 equality} does not apply anymore,
so a different argument is needed.  The ``only if'' part of the claim is clear.
We now prove the if part.
Let $\gamma:A \to C$ be the highest path within $\Reg(P)$ from $A \to C$,
and let $\delta:B \to D$ be the lowest possible path within $\Reg(P)$ from $B$ to~$D$.
Since the injection $\vk$ in Lemma~\ref{l:lattice path bijection 1} is now a bijection,
it follow that $\gamma$ and $\delta$ intersects at a point $E$.
If $E$ is contained in the segment $AB$ (resp.~$CD$), then
the segment $AB$ (resp.~$CD$) is contained in the lower (resp.~upper)
boundary of $\Reg(P)$ and thus every path counted here must pass through $E=A$
(resp.~$E=D$).  If $E$ is not contained in the segment $AB$ or~$CD$,
then $E$ is an intersection of the upper and lower boundary of \ts $\Reg(P)$,
and every path in \ts $\Reg(P)$ \ts must pass through~$E$.  This completes the proof.
\end{proof}

\bigskip

\section{Stanley's log-concavity} \label{sec:Sta}

Theorem~\ref{t:q-Sta} is a direct Corollary of Theorem~\ref{t:q-KS}
when setting $x$ to be a $\wh{0}$ element in the poset.
But its proof via lattice paths is much more direct, and illustrative,
so we discuss it separately here first.

\subsection{Proof of Theorem~\ref{t:q-Sta}}
Without loss of generality, assume \ts $x \in \Cr_1$, so $x=\al_r$ for some~$r$.
Let \. $Y^{\<k\>} = (r-1,k-r)$, so that the lattice paths corresponding to
linear extensions $L$ with \ts $L(x)=k$ \ts pass through \. $A_1':=Y^{\<k\>}$ \. and
\. $A':=Y^{\<k\>}+\eone$. Let \. $A:= Y^{\<k+1\>} +\eone  = A'+\etwo$, \.
$A_1 := Y^{\<k+1\>}$, \. $B_1:= Y^{\<k-1\>}$, \. $B := B_1+\eone$.
Then the paths with \ts $L(x)=k+1$ \ts pass through $A_1,A$
and the paths with \ts $L(x)=k-1$ \ts pass through $B_1,B$.
We can then write the difference between the left and right hand sides
of inequality~\eqref{eq:q-Sta-ineq} in terms of lattice paths as
\begin{equation}
\begin{split}
\Delta \, := \, \aN_q(k)^2 \, - \, \aN_q(k-1)  \.\cdot \.\aN_q(k+1) \,
= \, q^{2\binom{a+1}{2} + 2k -2r}  \, \times \qquad\qquad\qquad\\
 \times \. \biggl[ \aK_q(\zero, A_1')^2  \.\aK_q( A', Q)^2 \, - \,
\aK_q(\zero, B_1)  \.\aK_q(\zero, A_1) \. \aK_q( B, Q) \. \aK_q(A, Q)  \biggr]\..
\end{split}
\end{equation}

We now apply Lemma~\ref{l:lattice path bijection 1} twice as follows.
Let \ts $B_1'=A_1'$ \ts and \ts $C=D=\zero$. Observe that this configuration matches
the configuration in the Lemma by rotating \ts $\Reg(P)$ \ts by \ts $180^\circ$.
Note that we can apply the lemma since \. $\oa{A_1A_1'}=-\oa{B_1B_1'}=-\etwo$ \.
and \. $a_2'-b_2 =1 \geq 0 = c_2-d_2$.  Thus:
$$
\aK_q(\zero, A_1')^2 \ = \ \aK_q(\zero, A_1') \.\cdot \. \aK_q(\zero, B_1')
\ \geqslant \ \aK_q(\zero, A_1)  \.\cdot \. \aK_q(\zero, B_1).
$$
Similarly, on the other side we apply the lemma with \ts $A'=B'$ \ts and \ts $C=D=Q$,
satisfying the conditions since \. $\oa{AA'} = \etwo = -\oa{BB'}$ \.
and \. $a_2'-b_2=1 > 0=c_2-d_2$. Thus:
$$
\aK_q(A',Q)^2 \ = \ \aK_q( A',Q)  \.\cdot \.\aK_q( B',Q) \ \geqslant \ \aK_q( A,Q) \.\cdot \. \aK_q(B,Q).
$$
Multiplying the last two inequalities we obtain the desired inequality \ts $\Delta \geq 0$.
\qed


\begin{figure}
\includegraphics[width=2in]{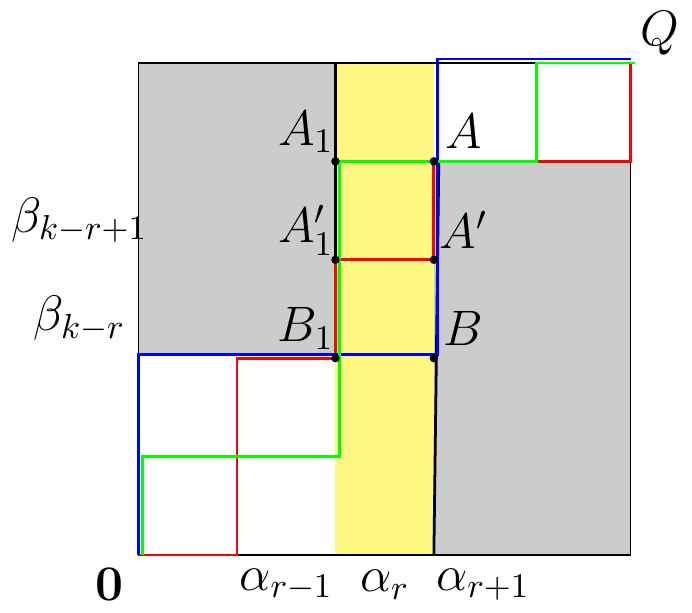}
\caption{The equality case in Stanley's Theorem~\ref{t:q-Sta} leading to the statement in Theorem~\ref{t:q-Sta-equality}, for the element $x=\alpha_r$ and $A_1'=B_1' = Y^{\<k\>}$. The lattice paths equality Lemma~\ref{l:bijection 1 equality} (after $180^\circ$ rotation) implies that all paths passing through $A_1$ also pass through $B_1$, so $A_1B_1$ is part of the upper boundary of $\Reg(P)$ (shaded in gray). Similarly for the paths passing through $A,B$. We see that the square centered at $(r-1+\frac12,k-r+\frac12) \not \in \Reg(P)$, which means that $\al_r \.||\. \beta_{k-r+1}$, and similarly we derive the other conditions.} \label{fig:sta-equality}
\end{figure}

\medskip

\subsection{Proof of Theorem~\ref{t:q-Sta-equality}}
It is clear that \eqref{itemequality Stanley d} $\Rightarrow$ \eqref{itemequality Stanley c}, \eqref{itemequality Stanley d} $\Rightarrow$ \eqref{itemequality Stanley b}, \eqref{itemequality Stanley c} $\Rightarrow$ \eqref{itemequality Stanley a},
and \eqref{itemequality Stanley b} $\Rightarrow$ \eqref{itemequality Stanley a}.
We now show that \eqref{itemequality Stanley a} $\Rightarrow$ \eqref{itemequality Stanley d}.	
In the proof of the Stanley inequality, notice that the equality is achieved exactly
when all applications of Lemma~\ref{l:lattice path bijection 1} lead to equalities.
For the equality in the first application of Lemma~\ref{l:lattice path bijection 1}, we have:
$$
\aK_q(\zero, A_1') \, \aK_q(\zero, B_1') \ = \ \aK_q(\zero, A_1) \, \aK_q(\zero, B_1).
$$
This equality case is covered by Lemma~\ref{l:bijection 1 equality}
(after $180^\circ$ rotation), which implies that the segment \ts $A_1B_1$ \ts
is part of the upper boundary of \ts $\Reg(P)$
 (which is the condition after rotating by $180^{\circ}$).
 The second application of Lemma~\ref{l:lattice path bijection 1}
 implies that~$AB$ is part of the lower boundary of \ts $\Reg(P)$.
 Thus every path \ts $\zero \to Q$ \ts passes on or below $B_1$ and on or above~$A$.
 Hence \. $q\aN_q(k-1) = \aN_q(k) = q^{-1}\aN_q(k-1)$,
 where the factors of~$q$ arise from the different horizontal levels
 of the path passing  from the \ts $A_1B_1$ \ts segment to the $AB$ segment.

We now show that \eqref{itemequality Stanley a} $\Rightarrow$ \eqref{itemequality Stanley e}.
Since the lattice paths and \ts $\Reg(P)$ \ts correspond to the poset structure,
we can restate the above conditions on poset level. The fact that \ts $A_1B_1$ \ts
is an upper boundary of \ts $\Reg(P)$ \ts implies that the element \.
$\beta_{k-r} \succ \alpha_{r-1}$. The fact that \. $BB_1, AA_1 \subset \Reg(P)$ \.
implies that \. $\beta_{k-r}, \beta_{k-r+1}$ \. are not comparable to~$\alpha_r$.
Finally, $AB$ on the lower boundary of \ts $\Reg(P)$ \ts implies \.
$\alpha_{r+1} \succ \beta_{k+1-r}$.

We now show \eqref{itemequality Stanley e} $\Rightarrow$ \eqref{itemequality Stanley b}
(cf.\ Proposition~\ref{prop:equality-KS-c->b}  for a proof of the analogous implication
for Kahn--Saks equality for general posets).
 Denote \ts $\cN(i) :=\{L\in \Ec(P)~:~L(x)=i\}$, so that \ts $\aN(i)=|\cN(i)|$.
Let $L \in \cN(i)$.
It follows from \eqref{itemequality Stanley e} that \ts $L(\beta_{k-r})=k-1$ \ts  and  \ts $x  \. || \. \beta_{k-r}$.
Thus there is an injection \. $\cN(k) \to \cN(k-1)$ \. by relabeling \ts $x \leftrightarrow \beta_{k-r}$,
 so that \ts $L(x)=k-1$ and $L(\beta_{k-r})=k$.  Thus, \ts $\aN(k) \le \aN(k-1)$.  Similarly, we obtain
 \ts $\aN(k) \le \aN(k+1)$ \ts  by relabeling \ts $x \leftrightarrow \beta_{k-r+1}$.  However, by the Stanley
 inequality (Theorem~\ref{t:Sta}), we have \. $\aN(k)^2 \ts \ge \ts \aN(k-1)\. \aN(k+1)$, implying
 that all inequalities are in fact  equalities.
\qed

\bigskip

\section{Proof of Theorem~\ref{t:q-KS}}\label{s:q-KS-proof}

For a given integer $w\in \nn$,
let \ts $\aF_q(w;k)$ \ts be the $q$-weighted sum
\begin{align*}
\aF_q(w;k) \ := \  \sum_L \, q^{\wgt(L)}\,,
\end{align*}
where the sum is over all linear extensions \ts $L\in \Ec(P)$,
such that \ts $L(x)=w$ \ts and \ts $L(y)=w+k$.  By definition,
\begin{align*}
\aF_q(k) \ = \ \sum_{w \in \nn} \,  \aF_q(w;k)\ts.
\end{align*}
We can thus express the difference
\begin{equation}\label{eq:KS-diff-sums}
\begin{split}
\Delta \ :&= \ \aF_q(k) \.\cdot\.\aF_q(k)  \ - \ \aF_q(k-1) \.\cdot\.\aF_q(k+1) \\
&= \ \sum_{v, \, v' \in \Zb} \, \aF_q(v;k) \.\cdot\. \aF_q(v';k) \ - \  \aF_q(v;k-1) \.\cdot\.\aF_q(v';k+1) \\
&= \  \sum_{u>w-1} \. S(u;w) \ + \ \frac12 \. \sum_{u=w-1} \. S(u;w),
\end{split}
\end{equation}
where we have grouped the terms in the expansions of products of \ts $\aF_q(*;*)$ \ts using
\begin{equation}\label{eq:Suw}
\begin{split}
S(u;w) \ = & \ \aF_q(u;k) \.\cdot \. \aF_q(w;k)  \, + \, \aF_q(w-1;k) \.\cdot \.\aFr_q(u+1;k) \ \\
&- \ \aF_q(u;k+1) \.\cdot \. \aF_q(w;k-1) \, - \, \aF_q(u+1; k-1)\.\cdot \. \aF_q(w-1;k+1).
\end{split}
\end{equation}
In order to verify the identity~\eqref{eq:KS-diff-sums}, let \ts $u\geq  w-1$.
Note that by setting \. $v \leftarrow w$, \. $v'\leftarrow u$ \. into the first term,
and setting \. $v \leftarrow u+1$, \. $v' \leftarrow w-1$ \. into the second term of~\eqref{eq:Suw},
 we cover the cases \ts $v' \geq v-1$ \ts and \ts $v \geq v'+1$ \ts in the positive summands
 in~\eqref{eq:KS-diff-sums}, where the double appearance of $v'=v-1$ is balanced out by the factor $\frac{1}{2}$. Similarly, for the negative terms, setting \. $v' \leftarrow u$, \.
 $v \leftarrow w$ \. covers the terms \ts $v' \geq v-1$, and setting \ts $v' \leftarrow w-1$,
 \ts $v \leftarrow u+1$ \ts covers the terms \ts $v-1 \geq v'$.
 Formally, we have:
$$
\aligned
& \sum_{u>w-1} \, \aF_q(u;k) \.\cdot \. \aF_q(w;k)  \, + \, \aF_q(w-1;k) \.\cdot \. \aF_q(u+1;k) \\
& \qquad  = \ \sum_{v' \geq v} \, \aF_q(v;k) \.\cdot \. {}\aF_q(v^\prime;k)
\ + \ \sum_{v \geq v^\prime+2} \, \aF_q(v;k) \.\cdot \. \aF_q(v^\prime;k)
\endaligned
$$
Similarly, we have:
\begin{align*}
& \sum_{u>w-1} \, \aF_q(u;k+1) \.\cdot \.  \aF_q(w;k-1) \, + \, \aF_q(u+1; k-1)\.\cdot \.  \aF_q(w-1;k+1) \\
& \qquad = \ \sum_{v' \geq v}\, \aF_q(v';k+1) \.\cdot \.  \aF_q(v;k-1) \  + \ \sum_{v \geq v'+2} \, \aF_q(v;k-1)\.\cdot \.  \aF_q(v';k+1),
\end{align*}
and the remaining case of \ts $v'=v-1$ \ts comes from \. $\frac12 S(u;u+1)$.

%
%

\begin{figure}
\includegraphics[width=5in]{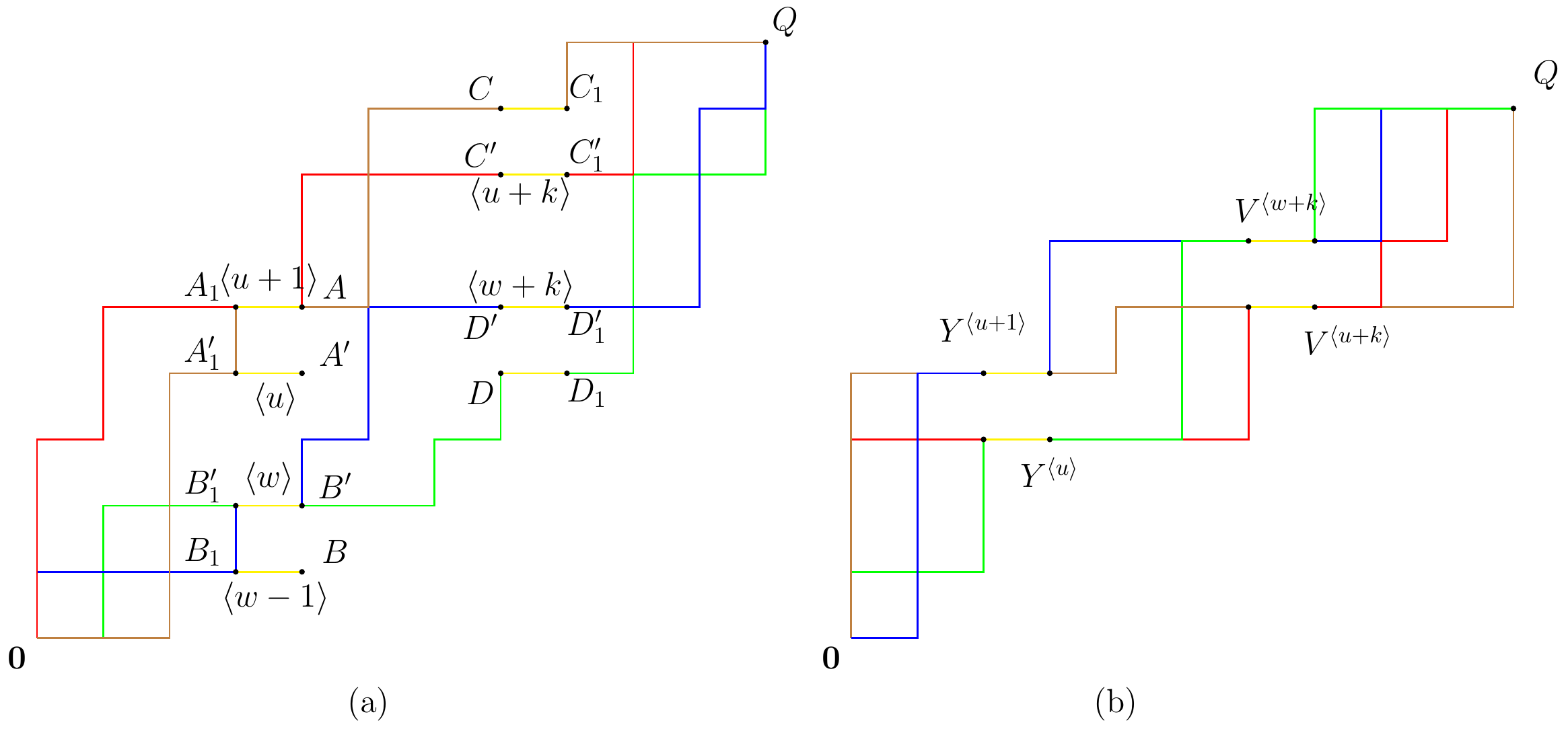}
\caption{The proof of Theorem~\ref{t:q-KS} in the case when $x,y\in \Cr_1$. Not all paths in $S(u;w)$ are drawn to avoid overcluttering. Figure (b) shows the paths involved in $S(u;u+1)$, which is the difference between the $q$-weighted sum of (red, blue) pairs minus the $q$-weighted sum of (green, brown) pairs.
 }\label{fig:KS-proof}
\end{figure}

\smallskip

We now prove that \. $S(u;w) \geqslant 0$ \. for all \ts $u \geq w-1$ \ts appearing in~\eqref{eq:KS-diff-sums}.
 Suppose \. $x,y \in \Cr_1$ \. so \. $x=\alpha_s$, and \. $y=\alpha_{s+r}$. 
For \ts $u\in \Zb$, let \. $\Yu := (s-1, u-s)$ \. and \. $\Vu: = (s+r-1,u-(s+r))$, that is, if a linear extensions has $L(x)=w$ then its lattice path passes through $Y^{\<w\>}, Y^{\<w\>}+\eone$, and if $L(y) = w+k$ then it passes through $V^{\<w+k\>},V^{\<w+k\>}+\eone$.

 In terms of lattice paths, we have:
\begin{align*}
\aF_q(w;k) \ = \  q^{\binom{\ana+1}{2}+2w+k} \, \aK_q\big(\zero, Y^{\<w\>}\big) \, \aK_q\bigl(Y^{\<w\>}+\eone,V^{\<w+k\>}\bigr) \,  \aK_q\bigl(V^{\<w+k\>}+\eone,Q\bigr) .
\end{align*}
Let first $u>w-1$ and for simplicity label the following points $A_1 = Y^{\<u+1\>}$, $A= Y^{\<u+1\>} + \eone$, $B_1 = Y^{\<w-1\>}$, $B=B_1+\eone$ and their shifts by $\pm \etwo$ as $A'_1 = \Yu $, $A' = A - \etwo = A'_1 + \eone$, $B_1' = \Yw$ and $B'=B_1' + \eone = B+\etwo$. Similarly, let $C=V^{\<u+k+1\>}$, $C_1 = C+\eone$, $C'= C-\etwo = V^{\<u+k\>}$, $C_1' = C_1 -\etwo$ and $D=V^{\<w+k-1\>}$, $D' = D+\etwo$ , $D_1 = D+\eone$ and $D_1' = D_1 +\etwo$.

Thus, letting \. $\ell = 2\binom{a+1}{2} + 2u+2w+2k$, we can expand $S(u;w)$ and regroup its terms as follows:
\begin{equation}\label{eq:S_delta}
\begin{split}
S(u;w) \. q^{-\ell} \ = &  \ \aK_q(\zero,A_1) \, \aK_q(\zero,B_1) \, \aK_q(A,C) \, \aK_q(B,D) \, \aK_q(C_1,Q) \, \aK_q(D_1,Q)\\
 & \quad + \,  \aK_q(\zero,A'_1) \, \aK_q(\zero,B'_1)  \, \aK_q(A',C') \, \aK_q(B',D') \,  \aK_q(C_1',Q) \,  \aK_q(D_1',Q)\\
&\quad - \, \aK_q(\zero,A_1) \, \aK_q(\zero,B_1)  \, \aK_q(A,C') \, \aK_q(B,D')  \, \aK_q(C'_1,Q)  \, \aK_q(D'_1,Q) \\
&\quad - \, \aK_q(\zero,A'_1) \, \aK_q(\zero,B'_1)  \, \aK_q(A',C) \, \aK_q(B',D)  \, \aK_q(C_1,Q)  \, \aK_q(D_1,Q) \\
\ = &  \ \Delta_1(\zero;A_1/A'_1,B'_1/B_1)  \, \Delta_1(A',B';C/C',D'/D;Q) \\
& \quad +   \, \aK_q(\zero,A_1)  \, \aK_q(\zero,B_1) \,  \Delta_1(A/A',B'/B;C',D')  \,  \Delta_1(C/C',D'/D;Q) \\
& \quad +  \, \aK_q(\zero,A_1) \, \aK_q(\zero,B_1) \,  \Delta_2(A,B;C,D) \,  \aK_q(C_1,Q) \, \aK_q(D_1,Q).
\end{split}
\end{equation}
Here the $\Delta$ notation means that we take differences of paths passing through either $E$ or $E'$ when using the $E/E'$, and $\Delta_2$ plays the role of a second derivative.
Specifically, the restructured terms above are given as follows, they are each nonnegative by our lattice paths lemmas:
\begin{align*}
\Delta_1(\zero;A_1/A'_1,B'_1/B_1) \ &: = \ \aK_q(\zero,A_1')\, \aK_q(\zero,B_1') \,- \,\aK_q(\zero,A_1)\,\aK_q(\zero,B_1) \ \geqslant_{\text{Lem}~\ref{l:lattice path bijection 1}} \ 0,\\
\Delta_1(A',B';C/C',D'/D;Q)\ &: = \ \aK_q(A',C',Q)\,\aK_q(B',D',Q)\, -\, \aK_q(A',C,Q)\, \aK_q(B',D,Q)  \ \geqslant_{\text{(see below)}} \ 0,\\
\Delta_1(C_1/C'_1,D'_1/D_1;Q) \ &: = \  \aK_q(C'_1,Q) \,\aK_q(D'_1,Q) \, -\,  \aK_q(C_1,Q) \,\aK_q(D_1,Q) \ \geqslant_{\text{Lem}~\ref{l:lattice path bijection 1}} \ 0,\\
\Delta_1(A/A',B'/B;C',D') \ &: = \  \aK_q(A',C')\,\aK_q(B',D') \, -\,  \aK_q(A,C') \,\aK_q(B,D')\ \geqslant_{\text{Lem}~\ref{l:lattice path bijection 1}} \ 0,\\
\Delta_2(A,B;C,D) \ &: = \ \aK_q(A,C)\, \aK_q(B,D)\, + \,  \aK_q(A',C')\,\aK_q(B',D')  \\
\qquad &\qquad  - \, \aK_q(A,C')\, \aK_q(B,D') \, - \, \aK_q(A',C)\,\aK_q(B',D) \ \geqslant_{\text{Lem}~\ref{l:path-averages}} \ 0.
\end{align*}
Here the second inequality follows by applying Lemma~\ref{l:lattice path bijection 1} twice:
$$
\aK_q(A',C') \ \aK_q(B',D') \ \geqslant \ \aK_q(A',C) \ \aK_q(B',D) \quad \text{ and } \quad \aK_q(C'_1,Q)\ \aK_q(D'_1,Q) \ \geqslant \  \aK_q(C_1,Q)\ \aK_q(D_1,Q).
$$

Now  let \ts $u=w-1$, we set \ts $Y':= \Yu+\eone$ \ts and \. $Y:= \Yw +\eone = Y^{\<u+1\>}+\eone = Y'+\etwo$,
and \. $V':=V^{\<u+k\>}$ \. and \. $V:=V^{\<u+k+1\>}=V'+\etwo$. Then:
\begin{equation}\label{eq: S lgv}
\aligned
\frac12 \. S(u;u+1) \ & = \, \aF_q(u;k) \, \aF_q(u+1;k)  \, - \, \aF_q(u;k+1) \, \aF_q(u+1;k-1)\\
& = \, q^{2\binom{\ana+1}{2}+4u+2+2k} \, \aK_q\big(\zero, \Yw\big) \, \aK_q\big(\zero, \Yu\big) \, \aK_q\big(V+\eone,Q\big)\, \aK_q\big(V'+\eone,Q\big) \ \times \\
& \qquad\qquad \times \ \biggl[ \aK_q(Y , V)\, \aK_q (Y',V') \, - \,  \aK_q(Y' ,V)\,\aK_q (Y, V') \biggr]
\  \geqslant_{\text{Lem~\ref{l: lgv}}} \  0\ts.
\endaligned
\end{equation}
This completes the proof. \qed
\bigskip

\section{Proof of Theorem~\ref{t:q-KS-equality}}\label{s:q-KS-equality}

\subsection{Setting up the proof}\label{ss:q-KS-setup}
It is clear that  \ts \eqref{itemequality q-linear} $\Rightarrow$ \eqref{itemequality q-KS}, \ts
\eqref{itemequality q-linear} $\Rightarrow$ \eqref{itemequality linear},
\ts \eqref{itemequality q-KS} $\Rightarrow$ \eqref{itemequality KS},
and \ts \eqref{itemequality linear} $\Rightarrow$ \eqref{itemequality KS}.

For \ts \eqref{itemequality combinatorial} $\Rightarrow$ \eqref{itemequality q-linear}, we adapt the proof of
Proposition~\ref{prop:equality-KS-c->b} below, of the analogous implication for general posets.
Without loss of generality, we assume that $z=x$ and $x \in \Cr_1$.
Then  \eqref{itemequality combinatorial}  implies that,
given a  linear extension \ts $L\in \Ec(P)$ \ts with \ts $L(y)-L(x)=k$, we can obtain
linear extension \ts $L'\in \Ec(P)$ \ts with \. $L'(y)-L'(x)=k-1$,
and linear extension \ts $L''\in \Ec(P)$ \ts with \. $L''(y)-L''(x)=k+1$,
by switching element~$x$ with the succeeding and preceding element in~$L$, respectively.
This map is clearly an injection that changes the $q$-weight by a factor of \ts $q^{\pm 1}$, so we have
\[
\aF_q(k-1)  \ \geqslant \ q \. \aF(k) \quad \text{and} \quad \aF_q(k+1)  \ \geqslant \ q^{-1} \. \aF(k).
\]
Since we also have \. $\aF_q(k)^2 \geq \aF_q(k-1) \aF_q(k+1)$ \. by Kahn--Saks Theorem~\ref{t:q-KS},
we conclude that equality occurs in the equation above, which proves~\eqref{itemequality q-linear}.

The proof of \ts \eqref{itemequality KS} $\Rightarrow$ \eqref{itemequality combinatorial} \ts will occupy
the rest of this section.  Together with the implications above, this implies the theorem.

\medskip

\subsection{Lattice paths interpretation} \label{ss:q-KS-paths}

Suppose  that $x=\alpha_s$ and $y=\alpha_{s+r}$.
We will assume without loss of too much generality that $r>1$,
so that the boundary between the region of $x$ and the region of $y$ does not overlap.
This allows us to apply the combinatorial interpretation in Lemma~\ref{l:bijection 1 equality} and Lemma~\ref{l: lgv}.
We remark that the method described here still applies to the case $r=1$ (by a slight modification of Lemma~\ref{l:bijection 1 equality} and Lemma~\ref{l: lgv}), and we  omit the details here for brevity.

The idea of the proof is as follows.  Informally, we will show that condition (a) implies that the regions above $x$ or $y$ is a vertical strip of width 1, that is the upper and lower boundary above $x$ and above $y$ are at distance 1 from each other, see Figure~\ref{fig:sta-equality}. These strips extend to the levels for which there exist a linear extension
\ts $L\in \Ec(P)$  \ts with \ts $L(y)-L(x)=k\pm 1$ (see the full proof for precise description in each possible case). In order to show this, we analyze the proof of Theorem~\ref{t:q-KS} in Section~\ref{s:q-KS-proof}. In order to have equality we must have \ts $S(u;w)=0$  \ts for every \ts $u \geq w-1$.  So we apply the equality conditions from Lemmas~\ref{l:bijection 1 equality} and~\ref{l: lgv} for every inequality involved in the proofs of \ts $S(u;w) \geq 0$. These equality conditions impose restrictions on the boundaries of  \ts $\Reg(P)$, making them vertical at the relevant levels above $x$ and $y$, and ultimately  drawing the width-1 vertical strip. This analysis requires choosing special points \ts $u$ and~$w$ from Section~\ref{s:q-KS-proof}, and the application of the equality Lemmas requires certain conditions. Thus there are several different cases which need to be considered.

In order to apply this analysis we parameterize  \ts $\Reg(P)$  \ts above $x$ and~$y$ as follows.
Let $u_0$ be the smallest possible value \. $L(x)=L(\al_s)$ \. can take, i.e.\
\. $(Y^{\<u_0\>},  Y^{\<u_0\>}+\eone)$ \. is a segment in the lower boundary of \ts $\Reg(P)$,
see Figure~\ref{fig:KS-equality}.
Let \ts $u_1-1$ \ts be the largest possible value that \ts $L(\al_{s-1})$ \ts can take, i.e.\ \.
$(Y^{\<u_1\>}-\eone, Y^{\<u_1\>})$ \. is a segment in the upper boundary of  \ts $\Reg(P)$.
Let   \ts $u_2+1$  \ts be the smallest possible value \ts $L(\al_{s+1})$ \ts can take,
i.e.\ \. $(Y^{\<u_2\>}+\eone, Y^{\<u_2\>}+2\eone)$ \. is a segment in the lower boundary of  \ts $\Reg(P)$.
Finally, let $u_3$ be the largest possible value $L(x)$ can take, so \.
$(Y^{\<u_3\>}, Y^{\<u_3\>}+\eone)$ \. is a segment in  the upper boundary of~$\Reg(P)$.
Clearly we have \ts $u_0 \leq u_1$ \ts and \ts $u_2 \leq u_3$.
Similarly, let \ts $w_0+k$ \ts be the smallest possible value  \ts $L(y)$  \ts can take, so this gives the level of the lower boundary of $\Reg(P)$ above~$y$. Finally, let  \ts $w_1+k-1$  \ts be the largest possible value  \ts $L(\alpha_{r+s-1})$  \ts can take, let  \ts $w_2+k+1$  \ts be the smallest possible value \ts $L(\alpha_{r+s+1})$  \ts can take, and  \ts $w_3+k$  \ts be the largest possible value  \ts $L(y)$ \ts  can take.
Clearly, we have \. $w_0\leq w_1$ \. and \. $w_2 \leq w_3$.

Here we are only concerned with \emph{effectively possible} values of~$u$,
i.e.\  values for which there exist linear extensions with  \ts $L(x)=u$ \ts
and \ts $L(y)-L(x) \in [k-1,k+1]$. We can thus restrict our region above $x$ and~$y$,
as follows. If we had  \ts $w_0 -u_0>1$, then  \ts $\aF(u_0;j) =0$  \ts for  \ts $j\in \{k-1, k, k+1\}$,
since $L(y) \leq u_0+k+1< w_0+k$. Thus we can assume that the region above $x$ starts at  \ts $L(x)=w_0-1$. Similarly, if  \ts
$w_0-u_0<-1$, we can restrict the region above $y$ accordingly. Thus we can assume \.
$|w_0-u_0| \leq 1$. Similarly, we can apply the same argument to the upper boundaries,
and assume that \. $|w_3-u_3|\leq 1$.  Finally,
let \ts $\vmax$ \ts be the largest integer such that  \ts $\aF(\vmax;k)>0$,
and let \ts $\vmin$ \ts be the smallest integer such that  \ts $\aF(\vmin;k)>0$.
Note that  \ts  $\vmax=\min\{u_3,w_3\}$  \ts and \ts  $\vmin=\max \{u_0,w_0\}$.

\begin{figure}[ht!]
	\includegraphics[width=2in]{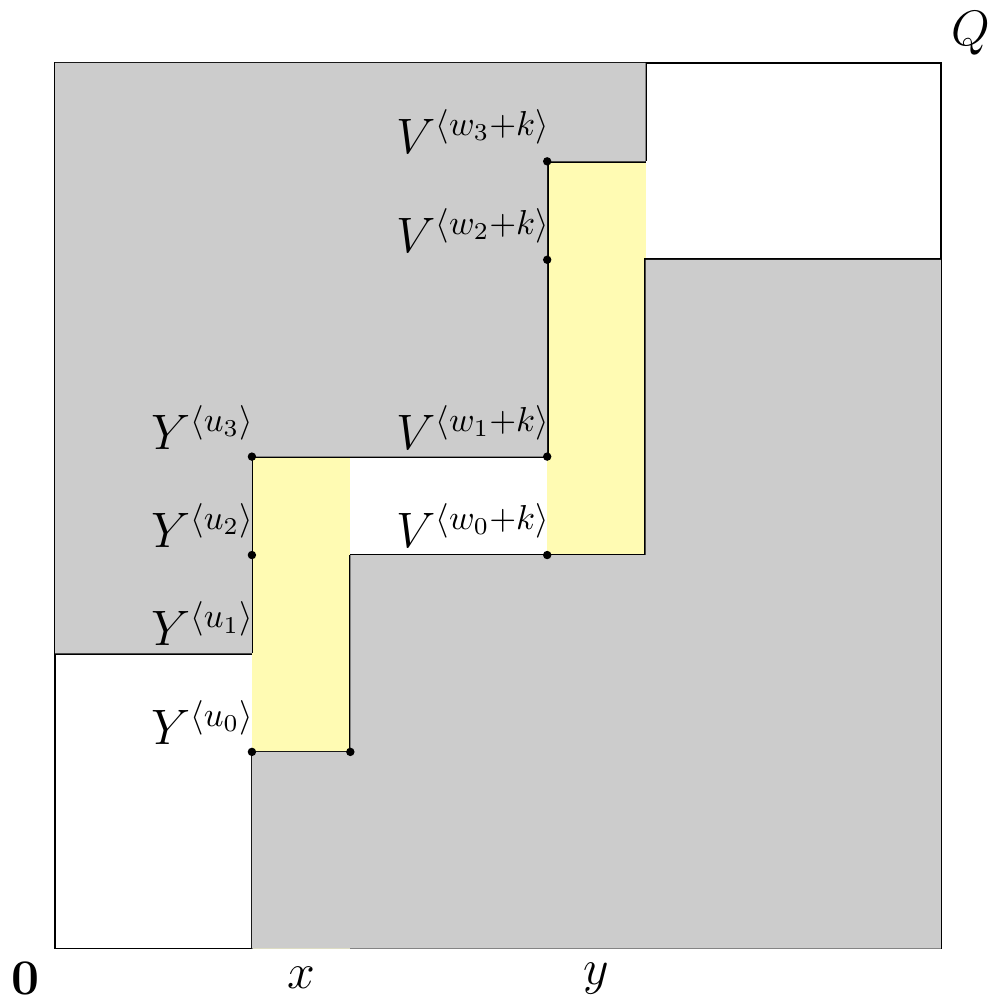}
	\caption{The structure of $\Reg(P)$ in the analysis of the Kahn--Saks equality.
		Here  \ts $k=5$, $u_0=5$, $u_1=6$, $u_2=7$, $u_3=8$,  \ts and  \ts  $w_0=5$, $w_1=6$, $w_2=8$, $w_3=9$.
	}\label{fig:KS-equality}
\end{figure}

\nin
In the language of lattice paths,
condition \eqref{itemequality combinatorial} follows from showing either of the following:
\begin{enumerate}
	[{label=\textnormal{({S\arabic*})},
		ref=\textnormal{S\arabic*}}]
	\item \label{itemS1}
	For every \. $v \in [\vmin, \vmax]$,
	we have $Y^{\<v-1\>}$ is contained in the upper boundary of $\Reg(P)$,
	and
	 $Y^{\<v+1\>} +\eone$ is contained in the lower boundary of $\Reg(P)$.
	\item \label{itemS2}
	For every \. $v \in [\vmin, \vmax]$,
		we have  $Y^{\<v+k-1\>}$ is contained in the upper boundary of $\Reg(P)$,
	and $Y^{\<v+k+1\>} +\eone$ is contained in the lower boundary of $\Reg(P)$.
\end{enumerate}

\nin
Note that these condition imply the width-1 vertical strip above $x$ or $y$ for all relevant values. It also implies that \. $L^{-1}(v \pm 1)\in \Cr_2$ \. and that \. $L^{-1}(v \pm 1) \. || \. x$ \.  since \. $Y^{\<v\pm 1\>}, \. Y^{\<v\pm 1\>} + \eone \in \Reg(P)$.

 For the rest of the section,
 we write \ts
 $A':=Y^{\<u\>}$, \ts $B':= Y^{\<w\>}$, and let the notations \.
   $A,B,A_1, B_1, A_1', B_1'$, $C,D, C', D', C_1, D_1, C_1',D_1'$ \. be  as in the proof of Theorem~\ref{t:q-KS} in  Section~\ref{s:q-KS-proof}. The choices of $u$ and $w$ will be chosen separately for each case of consideration.
   We also write \ts $m:=u_3-u_0$ \ts and \ts $m':=w_3-w_0$.

\smallskip

We split the proof  into different cases, depending on the values of \.
 $m$, $m'$, $u_1-u_0$,  $w_3-w_2$, $u_0-w_0$, and \ts $u_3-w_3$.

\medskip

\subsection{The cases $m\geq 2$, $u_0<u_1$ or  $m'\geq 2$, $w_2 <w_3$.}

We will now prove that \eqref{itemS2} holds for the first case.
The second case is analogous, after \ts $180^\circ$ \ts rotation of the configuration,
and leads to~\eqref{itemS1}.

Note that  \ts $\aF(u_0+1;k) >0$  \ts since there is a linear extensions  \ts $L(x)=u_0+1$
\ts and \. $L(y)=u_0+k+1 \in[ w_0+k, w_3+k]$.
We then have:
\[
u_0  \ \leq \  \vmin  \ \leq \  u_0+1 \ \leq \ u_1 \qquad \text{and} \qquad u_0+1  \ \leq \  \vmax \ \leq u_3\..
\]

We now turn to the proof of the inequality in Section~\ref{s:q-KS-proof}, and
notice  that equality in~\eqref{eq:KS-ineq} would be achieved  only if  \ts $S(u;w)=0$.
Let \ts  $u:=\vmax$  \ts and  \ts $w:=\vmin$.
Since  \ts $S(u;w)=0$,  this means that
\begin{align*}
\Delta_1(\zero;A_1/A'_1,B'_1/B_1) \, =\, 0 \qquad &\text{or} \qquad
\Delta_1(A',B';C/C',D'/D;Q) \, = \, 0.
\end{align*}

 Now note that  by Lemma~\ref{l:bijection 1 equality} we must have \. $\Delta_1(\zero;A_1/A'_1,B'_1/B_1) >0$,
 since the condition of $A_1B_1$ being part of $\Reg(P)$'s boundary is not satisfied: $B_1 = Y^{\<w-1\>}$ is not part of the upper boundary of $\Reg(P)$ since $w \leq u_1$.  Thus we must have $\Delta_1(A',B';C/C',D'/D;Q) =0$. This  implies that
\begin{equation}\label{eqKSequality condition 1}
	\aK(A',C')\, \aK(B',D')\, \aK(C_1',Q)\, \aK(D'_1,Q) \ = \ \aK(A',C)\, \aK(B',D) \, \aK(C_1,Q)\, \aK(D_1,Q).
\end{equation}

Let us show that every terms in the left side of~\eqref{eqKSequality condition 1} is nonzero.
Suppose otherwise, that \ts $\aK(A',C')=0$ (the other cases are treated analogously).
By the monotonous boundaries of \ts $\Reg(P)$, we must have $A'$ or $C'$ not in
\ts $\Reg(P)$, contradicting the choice of $u$ since there are linear extensions
with \ts $L(x) = u$ \ts and  \ts $L(y) = y+k$.

Therefore, we must have equality in both applications of Lemma~\ref{l:lattice path bijection 1}, so we can apply the Equality Lemma~\ref{l:bijection 1 equality} to both terms in \eqref{eqKSequality condition 1} (one after $180^\circ$ rotation).
These equalities imply that \. $CD=Y^{\<\vmax+k+1\>} \, Y^{\<\vmin+k-1\>}$ \. is part of the upper boundary of $\Reg(P)$, and that \. $C_1D_1=(Y^{\<\vmax+k+1\>} +\eone) \, (Y^{\<\vmin+k-1\>}+\eone)$ \. is part of the lower boundary of $\Reg(P)$.
This implies~\eqref{itemS2}.

For the rest of the proof, we can assume that $w_2=w_3$ if $m'\geq 2$ and $u_0=u_1$ if $m\geq 2$.

\medskip

\subsection{The case $m\geq 2$,  $u_0=u_1$, $u_3>w_3$}

Since $u_3>w_3$, we have that $w_3= \vmax$ and $u_3= \vmax+1$.
Let $u:=\vmax$ and $w:=\vmax$.
Since $m\geq 2$ we have that $A_1,B_1 \in \Reg(P)$, and since $w_3<u_3$ we have that $CC_1 \not \in \Reg(P)$.
Thus we have:
\[
\aK(\zero,A_1), \, \aK(\zero,B_1), \, \aK(C_1',Q), \, \aK(D_1',Q) \ > \ 0
\qquad \text{and} \qquad \aK(A,C) =   \aK(A',C)  \ = \ 0.
\]
We will first show that the segment $AB$ is contained in the lower boundary of $\Reg(P)$.

Since $S(u;w)=0$, the vanishing of the second summand in~\eqref{eq:S_delta} implies that either
\begin{equation*}
	\aK(\zero,A_1)\.\cdot\. \aK(\zero,B_1)=0, \ \quad \text{or} \ \quad  \Delta_1(C_1/C'_1,D'_1/D_1;Q) = 0, \ \quad \text{or} \ \quad
	\Delta_1(A/A',B'/B;C',D')= 0.
\end{equation*}
The first product is nonzero from above.  Below we show that \ts $\Delta_1(C_1/C'_1,D'_1/D_1;Q) \ne 0$,
implying that \ts $\Delta_1(A/A',B'/B;C',D')= 0$.

Note that the expression for $S(u;w)$ is implicitly over paths containing the entire horizontal segments above $x,y$. That is, in equation~\ref{eq:S_delta}, there is a summand containing $\aK(*,C)$ if and only if it also contains $\aK(C_1,Q)$, because the whole expression counts paths passing through $CC_1$. Thus, we can replace $\aK(C_1,Q)$ everywhere by $\wh{\aK}(C_1,Q):=\aK(C,C_1,Q)$.
With this replacement we have that $\wh{\aK}(C_1,Q)=0$ since $C \not \in \Reg(P)$ and so:
\begin{align*}
	 \aK(C'_1,Q)  \. \cdot \. \aK(D'_1,Q) \ >\ 0   \ = \  \wh{\aK}(C_1,Q) \. \cdot \. \aK(D_1,Q).
\end{align*}
This implies that \ts $\Delta_1(C_1/C'_1,D'_1/D_1;Q)\ne 0$, and, therefore, \ts $\Delta_1(A/A',B'/B;C',D')=0$.
This in turn implies that $AB$ is contained in the lower boundary of \ts $\Reg(P)$ \ts
by the Equality Lemma~\ref{l:bijection 1 equality}.

Now note that,
since $AB$ is in the lower boundary of \ts $\Reg(P)$,
every path in \ts $\Reg(P)$ \ts must pass through \. $A=Y^{\<\vmax +1\>}+\eone = Y^{\<u_3\>}+\eone $.
Also note that, since $u_0=u_1$, we have \. $Y^{\<u_0\>}\, Y^{\<u_0+1\>}$ \. is in the upper boundary of \ts $\Reg(P)$,  so every path in \ts $\Reg(P)$ \ts must pass through
$Y^{\<u_0\>}$.
These two properties imply that paths differ only by the level of their horizontal segment above $x$ and so
\begin{equation}\label{eqFK calculation 1}
\begin{split}
		\aF(v;k-1) \ &= \ \aF(v-1;k) \quad \text{ for every } \quad v \in [u_0+1,u_3]\.,\\
	\aF(v;k+1) \ &= \ \aF(v+1;k) \quad \text{ for every } \quad v \in [u_0,u_3-2]\..
\end{split}
\end{equation}
We will use ~\eqref{eqFK calculation 1} to show that \ts $\vmin=u_0+1$.

Suppose first that $\vmin = u_0$.
Then \eqref{eqFK calculation 1} gives us
\begin{align*}
	\aF(k-1) \ &= \ \sum_{v=u_0+1}^{u_3} \, \aF(v;k-1) \ = \ \sum_{v=u_0+1}^{u_3} \, \aF(v-1;k) \ = \ \sum_{v=u_0}^{u_3-1} \, \aF(v;k) \ =  \ \aF(k),\\
	\aF(k+1) \ &= \  \sum_{v=u_0}^{u_3-2} \, \aF(v;k+1) \ = \ \sum_{v=u_0}^{u_3-2} \, \aF(v+1;k) \ = \ \sum_{v=u_0+1}^{u_3-1} \,\aF(v;k) \ <  \ \aF(k).
\end{align*}
So we have \. $\aF(k)^2 > \aF(k-1) \. \aF(k+1)$, a contradiction.

Then  suppose  that $\vmin = u_0-1$.
Then \eqref{eqFK calculation 1}   gives us
\begin{align*}
	\aF(k-1) \ &= \ \sum_{v=u_0}^{u_3} \aF(v;k-1) \ = \ \sum_{v=u_0+1}^{u_3} \aF(v-1;k)  \ + \ \aF(u_0;k-1) \ = \ \sum_{v=u_0}^{u_3-1} \aF(v;k)  \, + \, \aF(u_0;k-1) \\
	 &=  \ \aF(k) \ + \ \aF(u_0;k-1),\\[2 pt]
	\aF(k+1) \ &= \  \sum_{v=u_0}^{u_3-2} \aF(v;k+1) \ = \ \sum_{v=u_0}^{u_3-2} \aF(v+1;k) \ = \ \sum_{v=u_0+1}^{u_3-1} \aF(v;k) \ =  \ \aF(k) \, -  \, \aF(u_0;k).
\end{align*}
 On the other hand, since \ts $\vmin =u_0 -1$  \ts and  \ts $\vmax=u_3-1$, we then have \ts  $m'=m\geq 2$, so we can without loss of generality assume that  \ts $w_2=w_3$ from the conclusion of the previous subsection.
Since \ts  $w_2=w_3$, we then have:
\[\aF(u_0;k-1) \  \leq \  \aF(u_0;k).  \]
Combining these two equations, we then have
\begin{align*}
	\aF(k-1) \.\cdot\. \aF(k+1) \ = \ \big[\aF(k) +  \aF(u_0;k-1)\big] \.\cdot\. \big[\aF(k)  -  \aF(u_0;k)\big] \ \leq \ \aF(k)^2 \ - \  \aF(u_0;k)^2 \ < \ \aF(k)^2\.,
\end{align*}
which is another contradiction.
Hence, since \. $\vmin\in [u_0-1,u_0+1]$, we conclude that we must have  \ts $\vmin = u_0+1$.

Now recall that the combinatorial properties say that
\. $Y^{\<u_0\>}=Y^{\<\vmin-1\>}$ \. is contained in the upper boundary of \ts $\Reg(P)$,
and \. $Y^{\<\vmax+1\>}+\eone$ \. is contained in the lower boundary of \ts $\Reg(P)$.
This implies \eqref{itemS1}, as desired.

An analogous conclusion  can be derived for the case  $u_0>w_0$ by applying the same argument.
Finally,
by the 180$^\circ$ rotation, an analogous conclusion can be drawn
for the case \ts $u_3<w_3$ \ts and/or \ts $u_0< w_0$.
Hence for the rest of the proof we can assume that \ts $u_0=w_0$
\ts and \ts $u_3=w_3$ \ts if \ts $m\geq 2$.
\medskip

\subsection{The case $m\geq 2$, $u_0=u_1$, $w_2=w_3$, $u_0=w_0$, $u_3=w_3$}
\label{subsecequality no overlapping boundaries}

Note that in this case \. $m=u_3-u_0=w_3-w_0=m'$, \. $\vmin=u_0=w_0$ \. and \. $\vmax=u_3=w_3$.
We will show that this case leads to a contradiction.

\smallskip

\nin
{\bf Claim:} Either the segment $(Y^{\<\vmax\>}+\eone, Y^{\<\vmin\>}+\eone)$ is contained in the lower boundary of $\Reg(P)$,
or the segment $(V^{\<\vmax+k\>}, V^{\<\vmin+k\>})$ is contained in the upper boundary of $\Reg(P)$.

\smallskip

To prove the claim,
let first \ts $u:=\vmax-1$ \ts and \ts $w:=\vmax$.
Since \ts $S(u;u+1)=0$, we get from equation~\eqref{eq: S lgv} that
\begin{align*}
	\aK(Y , V)\.\cdot\. \aK (Y',V')\  =\  \aK(Y' ,V)\.\cdot\.\aK (Y, V'),
\end{align*}
where \. $Y= Y^{\<\vmax\>}+\eone$, \. $Y'= Y^{\<\vmax-1\>}+\eone$ \.
and \. $V= V^{\<\vmax+k\>}$, \. $V'= V^{\<\vmax+k-1\>}$.
It then follows from Special Equality Lemma~\ref{l: lgv} that
there exists a point $E$ for which every path counted here must pass through,
and there are three subcases:

\begin{enumerate}
\item $E$ is equal to $A:=Y^{\<\vmax\>}+\eone$ and is contained in the lower boundary of $\Reg(P)$,

\item $E$ is equal to  $D:=V^{\<\vmax\>}$ and is contained in the upper boundary of $\Reg(P)$,

\item  $E$ is contained in the upper and lower boundary of $\Reg(P)$ (which then necessarily intersect).
\end{enumerate}

\smallskip

\nin {\bf Case (iii).} Suppose that $E$ is contained in  the upper and lower boundary of $\Reg(P)$, and in particular every path in $\Reg(P)$ must pass through $E$.
We now change our choice of $u$ and $w$ to  $u:=\vmax-1$ and $w:= \vmin +1$.  Note that here \. $AB=(Y^{\<\vmax\>}+\eone, Y^{\<\vmin\>}+\eone)$ \. and  \. $CD=(V^{\<\vmax+k\>}, V^{\<\vmin+k\>})$.
Observe that from $m \geq 2$ we have $u \geq w$. It follows from $S(u;w)=0$ and equation~\eqref{eq:S_delta} that $\Delta_2(A,B;C,D)=0$.
Rewriting $\Delta_2(A,B;C,D)$ using the intersection point $E$, we get
\[
\Delta_2(A,B;C,D) \ = \ \big[\aK(A,E)\aK(B,E) - \aK(A',E) \aK(B',E)\big]
\.\cdot\. \big[\aK(E,C)\aK(E,D) - \aK(E,C')\aK(E,D')\big].
\]
One of the factors must be zero,   so suppose that
\[
\aK(A,E)\.\cdot\.\aK(B,E) \, - \, \aK(A',E)\.\cdot\. \aK(B',E) \ = \ 0.
\]
By applying the Equality Lemma~\ref{l:bijection 1 equality}, we then have that
$AB=(Y^{\<\vmax\>}+\eone, Y^{\<\vmin\>}+\eone)$ is contained in the lower boundary of \ts $\Reg(P)$, as desired.
The case
$$
\aK(E,C)\.\cdot\.\aK(E,D) \, - \, \aK(E,C')\.\cdot\.\aK(E,D')\ = \ 0.
$$
uses a similar argument.  In that case, we conclude that \.
$(V^{\<\vmax+k\>}, V^{\<\vmin+k\>})$ \. is contained in the upper boundary of \ts $\Reg(P)$ \ts
instead, which proves the claim.

\smallskip

\nin {\bf Case (i).} Suppose that
$E$ is equal to \. $A=Y^{\<\vmax\>}+\eone$
and is contained in the lower boundary of \ts $\Reg(P)$.
Then it follows that
the segment \. $(Y^{\<\vmax\>}+\eone,Y^{\<\vmin\>}+\eone)$ \.
is contained in the lower boundary of \ts $\Reg(P)$, as desired.


\smallskip

\nin {\bf Case (ii).}  Suppose that $E$ is equal to
$D=V^{\<\vmax+k-1\>}$
and is contained in the upper boundary of $\Reg(P)$.
This implies  that $(V^{\<\vmax+k\>}, V^{\<\vmax+k-1\>})$ is contained in the upper boundary of $\Reg(P)$.
By $180^\circ$ rotation and using the same argument,
we can without loss of generality also assume that
$(Y^{\<\vmin+1\>}+\eone, Y^{\<\vmin\>}+\eone)$ is contained in the lower boundary of $\Reg(P)$.
Now let $u:=\vmax-1$ and $w:= \vmin +1$.
It again follows from $S(u;w)=0$ that $\Delta_2(A,B;C,D)=0$.
Since $(V^{\<\vmax+k\>}, V^{\<\vmax+k-1\>})$ is contained in the upper boundary of $\Reg(P)$,
we have:
\[ \aK(A',C') \ = \ \aK(A',C), \qquad \aK(A,C') \ = \ \aK(A,C). \]
Since $(Y^{\<\vmin+1\>}+\eone, Y^{\<\vmin\>}+\eone)$ is contained in the lower boundary of $\Reg(P)$,
we have:
\[ \aK(B',D') \ = \ \aK(B,D'), \qquad \aK(B',D) \ = \ \aK(B,D). \]
It then follows that $\Delta_2(A,B;C,D)$ can be rewritten as
\begin{align*}
	\Delta_2(A,B;C,D)  \ &=  \ \aK(A,C)\.\cdot\. \aK(B,D) \, + \, \aK(A',C)\.\cdot\.\aK(B,D')
\\ & \qquad \quad  - \ \aK(A,C)\.\cdot\. \aK(B,D') \, -\, \aK(A',C)\.\cdot\. \aK(B,D) \\
	& = \ \big( \aK(A,C) \.-  \. \aK(A',C)\big) \, \big( \aK(B,D)\. - \.  \aK(B',D)\big).
\end{align*}
Without loss of generality, assume that $\aK(A,C)-  \aK(A',C) =0$.
This implies that the segment \.
$AA'=(Y^{\<\vmax\>}+\eone, Y^{\<\vmax-1\>}+\eone)$ \. is contained in the lower boundary of \ts $\Reg(P)$,
which in turn implies that \,
$(Y^{\<\vmax\>}+\eone, Y^{\<\vmin\>}+\eone)$  \.
is contained in the lower boundary of \ts $\Reg(P)$.
This concludes the proof of the claim.

\medskip

Applying the claim, let
$(Y^{\<\vmax\>}+\eone, Y^{\<\vmin\>}+\eone)$  be contained in the lower boundary of $\Reg(P)$, the other case are treated analogously.
Note that we also have that $(Y^{\<\vmax\>}, Y^{\<\vmin\>})$ is contained in the upper boundary of $\Reg(P)$ since $u_0=u_1$.
 This implies that
 \begin{align*}
 	\aF(v;k+1) \ &= \ \aF(v+1;k)  \quad \text{ for every } v \in [\vmin, \vmax-1],\\
 	 \aF(v;k-1) \ &= \ \aF(v-1;k)  \quad \text{ for every } v \in [\vmin+1, \vmax].
 \end{align*}
We then have
\begin{align*}
	\aF(k+1) \ &= \ \sum_{v=\vmin}^{\vmax-1} \aF(v;k+1) \ = \ \sum_{v=\vmin}^{\vmax-1} \aF(v+1;k) \ = \ \sum_{v=\vmin+1}^{\vmax} \aF(v;k) \ < \ \aF(k),\\
	\aF(k-1) \ &= \ \sum_{v=\vmin+1}^{\vmax} \aF(v;k-1) \ = \ \sum_{v=\vmin+1}^{\vmax} \aF(v-1;k) \ = \ \sum_{v=\vmin}^{\vmax-1} \aF(v;k) \ < \ \aF(k).
\end{align*}
So we have $\aF(k)^2 > \aF(k-1) \aF(k+1)$, a contradiction. Hence this case does not lead to equality.

\medskip

\subsection{The case $m< 2$ and  $m'< 2$}
We now check the last remaining cases of Theorem~\ref{t:q-KS-equality}.

We first consider the case \ts $m=0$. We have \ts $L(x)=u=u_0=u_3$ \ts
is the unique possible value. Then, for every $k\in \nn$, we have:
 \[\aF(k) \ = \ \aN(k+u_3), \]
where \ts $\aN(j)$ \ts is the number of linear extensions
\ts $L\in \Ec(P)$ \ts for which \ts $L(y)=j$.
It then follows from the combinatorial description of
Theorem~\ref{t:q-Sta-equality} that \eqref{itemS2} holds.
By the same argument, we get an analogous conclusion for the case \ts $m'=0$.

\smallskip

We now consider the case \ts
$m=m'=1$.
First note that, if either \ts $w_0=u_0+1$ \ts or \ts $w_0=u_0-1$,
then we either have \ts $\aF(k-1)=0$ \ts or \ts $\aF(k+1)=0$,
which contradicts the assumption that \ts $\aF(k)>0$.
So we assume \ts $w_0=u_0$.
 Let \ts $u:=u_0$ \ts and \ts $w:=u_0+1$.
 By using \ts $S(u;u+1)=0$, from this part of the proof in Section~\ref{s:q-KS-proof},
 we have an application of Lemma~\ref{l: lgv}. By its equality criterion
 we see that there exists a point~$E$ for which every path counted here must pass through.
We now set for brevity
\[a \ := \  \aK(\zero, A_1, A, E), \quad b \ := \ \aK(\zero, B_1, B, E), \quad  c \ := \ \aK(E, C, C_1, Q), \quad d \ := \ \aK(E, D, D_1, Q).  \]
Using this notation, we have
\[ \aF(k) \ = \ ac\. + \.  bd, \quad \aF(k+1) \ = \ bc,  \quad \aF(k-1) \ = \ ad. \]
Then
\[ \aF(k)^2 \, - \, \aF(k+1) \.\cdot\. \aF(k-1) \ = \  (ac)^2 \. + \. (bd)^2 \. + \. acbd. \]
This equation is equal to zero only if \ts $ac= bd =0$,
which implies that \ts $\aF(k)=0$, a contradiction.
This completes the proof of \ts \eqref{itemequality KS} $\Rightarrow$ \eqref{itemequality combinatorial},
and finishes the proof of Theorem~\ref{t:q-KS-equality}. \qed

\bigskip

\section{Multivariate generalization}\label{sec:multi}

The $q$-weights in the introduction can be refined as follows.
Let \. $\bq:= (q_1,\ldots,q_{\ana})$ \. be formal variables.
Define the \defn{multivariate weight} of a linear extension \ts
$L \in \Ec(P)$ \ts as
$$
\bq^L \, := \, \prod_{i=1}^\ana \, q_i^{L(\al_i) \ts - \ts L(\al_{i-1})},
$$
where we set \ts $L(\al_0):=0$. In the language of lattice paths we see that the
power of \ts $q_i$ \ts is equal to one plus the number of vertical steps on the
vertical line passing through~$(i-1,0)$.

\medskip

\begin{thm}[{\rm \defn{Multivariate Stanley inequality}}]\label{t:bq-Sta}
Let \ts  $P=(X,\prec)$ \ts be a finite poset of width two,
let \ts $(\Cr_1,\Cr_2)$ \ts be the chain partition of~$P$, and let \ts $x\in \Cr_1$.
Define
$$
\aNr_\bq(k) \, := \, \sum_{L\in \Ec(P) \ : \ L(x)=k} \, \bq^{L}\..
$$
Then:
\begin{equation}\label{eq:bq-Sta-ineq}
\aNr_\bq(k)^2 \ \geqslant \ \aNr_\bq(k-1) \, \aNr_\bq(k+1) \quad \ \text{ for all} \ \quad k\.>\.1 \ts,
\end{equation}
where the inequality between polynomials in the variables \. $\bq=(q_1,\ldots,q_{\ana})$ \. is coefficient-wise.
\end{thm}

\medskip

When \ts $q_1=q_2=\ldots = q$, we obtain Theorem~\ref{t:q-Sta}.
Similarly, the following result generalizes both Theorem~\ref{t:q-KS} and Theorem~\ref{t:bq-Sta}.

\medskip

\begin{thm}[{\rm \defn{Multivariate Kahn--Saks inequality}}]\label{t:bq-KS}
Let \ts  $P=(X,\prec)$ \ts be a finite poset of width two,
let \ts $(\Cr_1,\Cr_2)$ \ts be the chain partition of~$P$, and let
\ts $x,y\in \Cr_1$ \ts be two distinct elements.  Define:
$$
\aFr_\bq(k) \ := \ \sum_{L\in \Ec(P) \ : \ L(y)-L(x)=k} \, \bq^{L}\..
$$
Then:
\begin{equation}\label{eq:bq-KS-ineq}
\aFr_\bq(k)^2 \, \geqslant \, \aFr_\bq(k-1) \. \aFr_\bq(k+1) \quad \ \text{for all} \ \quad k\. > \. 1\ts,
\end{equation}
where the inequality between polynomials in the variables $\bq=(q_1,\ldots,q_{\ana})$ is coefficient-wise.
\end{thm}

\medskip

For the proof, note that in the case \ts $x,y \in \Cr_1$,
the lattice paths lemmas in Subsections~\ref{ss:lemmas-ineq} and~\ref{sec:toolkit-criss-cross}
rearrange and reassign pieces of paths via vertical translation. Thus, we preserve the total
number of vertical segments above each \ts $(i,0)$ \ts in each pair of paths.
Therefore, the resulting injections preserve the multivariate weight \ts $\bq^L$, and
both theorems follow.  We omit the details.

\medskip

\begin{rem} {\rm
Note that, in general, this function is not quasi-symmetric in $q_1,q_2,\ldots$\., much less symmetric.
This generalization is different from the \emph{quasisymmetric functions} associated to $P$-partitions,
see e.g.~\cite[\S 7.19]{Sta}.  Still, the multivariate polynomials in the theorems can be
expressed in terms of the (usual) \emph{symmetric functions} in certain cases.

For example, let \ts $P$ \ts be the parallel product of two chains \ts
$\Cr_1$ \ts and \ts $\Cr_2$ \ts
of sizes \ts $\ana$ and $\bnb$, respectively.  Clearly, \ts $e(P)=\binom{\ana+\bnb}{\ana}$ \ts
in this case.  Fix \ts $x=\al_s$ \ts and \ts $y=\al_{r+s}$.
Then we have:
$$
\aF_\bq(k) \, = \, \sum_j \, h_j(q_1,\ldots,q_s) \, h_{k-r}(q_{s+1},\ldots,q_{s+r}) \, h_{\bnb-k+r-j}(q_{r+s+1},\ldots,q_{\ana}),
$$
where  \. $h_{i}(x_1,\ldots,x_k)$  \. is the \emph{homogeneous symmetric function} of degree~$i$, see e.g.~\cite[\S 7.5]{Sta}.
Similarly, from Section~\ref{s:q-KS-proof}, we have:
$$\aligned
\frac12 \. S(u;u+1) \, & = \, h_u(q_1,\ldots,q_s) \, h_{u+1}(q_1,\ldots,q_s) \, h_{k-1-r-u}(q_{s+r+1},\ldots,q_\ana) \, \times \\
& \hskip1.cm \times \,
h_{k-r-u}(q_{s+r+1},\ldots,q_\ana) \, s_{(k-r)^2}(q_{s+1},\ldots,q_{s+r}).
\endaligned
$$
The $\Delta$ terms involved in the other $S(u;w)$ can be similarly expressed in terms of
\emph{Schur functions}~$s_\la$ as in the formula above.
We leave the details to the reader.
} \end{rem}

\bigskip

\section{General posets}\label{sec:posets}

\subsection{Equality conditions in the Stanley inequality}\label{ss:posets-sta}
As in the introduction, let \ts $P=(X,\prec)$ \ts be a poset on $n$ elements.
Denote by \ts $f(u) := \bigl|\{v \in X\,{}:\,{}v\prec u\}\bigr|$ \ts and \ts
$g(u) := \bigl|\{v \in X\,{}:\,{}v\succ u\}\bigr|$ \ts the sizes of lower
and upper ideals of \ts $u\in X$, respectively, excluding the element~$u$.

\smallskip

\begin{thm}[{\rm \defn{Equality condition for the Stanley inequality}~\cite[Thm~15.3]{SvH}}]\label{t:Sta-equality-gen}
Let \ts  $P=(X,\prec)$ \ts be a finite poset, and let \ts $x\in X$.
Denote by \ts $\aNr(k)$ \ts the number of linear extensions \ts $L\in \Ec(P)$,
such that \ts $L(x)=k$.  Suppose that \ts $k \in \{1,\ldots,n-1\}$ \ts and \ts $\aNr(k)>0$. Then the following are equivalent:
\begin{enumerate}
			[{label=\textnormal{({\alph*})},
		ref=\textnormal{\alph*}}]
\item \ $\aNr(k)^2 \. = \. \aNr(k-1) \. \aNr(k+1)$,
\item \ $\aNr(k) \. = \. \aNr(k+1) \. = \. \aNr(k-1)$,
\item \ $f(y)>k$ \ts for all \ts $y\succ x$, and \ts $g(y)>n-k+1$, for all \ts $y\prec x$.
\end{enumerate}
\end{thm}

\smallskip

\begin{prop}\label{prop:pentagon}
For posets of width two, condition~{\rm (c)} in Theorem~\ref{t:Sta-equality-gen} is equivalent
to the $k$-pentagon property of~$x$, which is condition~{\rm (e)} in Theorem~\ref{t:q-Sta-equality}.
\end{prop}

\smallskip

The proof is a straightforward case analysis and is left to the reader.  Of course, the
proposition also follows by combining  Theorem~\ref{t:q-Sta-equality}
and Theorem~\ref{t:Sta-equality-gen}.

\smallskip

\begin{prop}[{\rm \cite[Lemma~15.2]{SvH}}]\label{prop:Sta-zero}
Let \ts $P=(X,\prec)$ \ts be a poset with $n$ elements,
let \ts $x\in X$ \ts and \ts $1\le k \le n$.  Then \ts
$\aNr(k)>0$ \ts if and only if \. $f(x)\le k-1$ \. and \. $g(x)\le n-k$.
\end{prop}

\smallskip

\begin{cor}\label{cor:Sta-positivity}
Let $P=(X,\prec)$ be a poset on \ts $|X|=n$ \ts elements, and let \ts $x\in X$.
Then, deciding whether \.
$\aNr(k)^2\ts =\ts \aNr(k-1) \ts \aNr(k+1)$ \. can be done in \ts {\rm poly$(n)$} \ts time.
\end{cor}

\smallskip

Here and everywhere below we assume that posets are presented in such a way that
testing comparisons \. ``$x \prec y$'' \. has $O(1)$ cost, so e.g.\ the function \ts $f(x)$ \ts
can be computed in \ts $O(n)$ \ts time.

\smallskip

\begin{proof}[Proof of Corollary~\ref{cor:Sta-positivity}]
Clearly, we have the equality for all \ts $\aN(k)=0$.  By Proposition~\ref{prop:Sta-zero},
this condition can be tested in polynomial time. Similarly,
condition~(c) in Theorem~\ref{t:Sta-equality-gen} implies that equality in the
Stanley inequality can be tested in polynomial time in the remaining cases.
\end{proof}

\medskip

\subsection{Vanishing conditions in the Kahn--Saks inequality}\label{ss:posets-KS-van}
The following result is a natural generalization of Proposition~\ref{prop:Sta-zero}.

\smallskip

\begin{thm}\label{t:F-positive}
Let \ts $P=(X,\prec)$ \ts and let \ts $x\prec y$, where \ts $x,y \in X$.
Denote
$$
h(x,y) \ := \ \big|\{u\in X\.{}:\.{}x\prec u \prec y\}\big|\ts.
$$
Then \ts $\aFr(k) >0$ \ts if and only if
$$h(x,y)  \, < \, k \, < \, n \ts - \ts f(x) \ts - \ts g(y).
$$

\end{thm}

\smallskip

\begin{proof}
For the \ts \emph{``only if''} \ts direction, let \ts $L\in \Ec(P)$ \ts be a linear extension such that
\ts $L(y)-L(x)=k$.  By definition,  we have \ts $f(x) \leq L(x)-1$ \ts and
\ts $g(y)\leq n-L(y)$, which implies
\[ f(x) + g(y) \, \leq \, L(x)-1 + n-L(y) \,  = \, n-k-1.   \]
Furthermore, condition \ts $L(y)-L(x)=k$ \ts implies
that \ts $h(x,y) \leq k-1$, as desired.

\medskip

For the \ts \emph{``if''} \ts direction, let \ts $c:=\min\{n-g(x), n-k-g(y)\}$.
Note that \ts $g(x) \leq  n-c$ \ts and \ts $g(y) \leq  n-c-k$.
%
We also have \ts  $f(x)  \leq  n-g(x) -1$ \ts  by definition of upper and lower ideals,
and \ts $f(x) \leq n-k-g(y)-1$ \ts by assumption.
Combining these two inequalities, we get \ts $f(x) \leq c-1$.

Since \ts $f(x) \leq c-1$ \ts and \ts $g(x) \leq n-c$, by Proposition~\ref{prop:Sta-zero},
there is a linear extension \ts $L\in \Ec(P)$ \ts such that \ts $L(x)=c$.
We are done if \ts $L(y)=c+k$, so suppose that \ts $L(y) \neq c+k$.
We split the proof into two cases.
	
\smallskip

\nin
$(1)$ \. Suppose  that \ts $L(y) < c+k$.
Since  \ts $g(y)  \leq \, n-c-k$,
there exists \ts $w \in X$ \ts such that \ts $w \. || \. y$ \ts and \ts $L(w) > L(y)$.
Let $w$ be such an element for which $L(w)$ is minimal, let \ts $a:=L(y)$ \ts and \ts $b:=L(w)$.
The minimality assumption implies that every  \. $u\in \{L^{-1}(a),\ldots, L^{-1}(b-1) \}$  \ts satisfies \ts $u \succ y$,
which gives \ts $u \. ||\. w$.

Define a new linear extension \ts $L'\in \Ec(P)$, obtained from $L$ by setting \ts
$$L'(w):= L(y), \quad L'(y) := L(y)+1, \quad L'(u) = L(u)+1 \ \  \text{for all} \ \ u \in X \ \ \text{s.t.} \ \ a\le L(u)\le b-1,
$$
and setting \ts $L'(v):=L(v)$ \ts for all other elements \ts $v\in X$. Note that \ts $L'(x)=L(x)$ \ts by definition.

Denote by \. $\Phi: L \to L'$ \. the resulting map on~$\Ec(P)$. From above, $\Phi$ \ts increases the difference
\ts $L(y)-L(x)$ \ts by one when defined.  Iterate~$\Phi$ until we obtain a
linear extension \ts $L^\diamond$ \ts that satisfies \. $L^\diamond(y)-L^\diamond(x)=(c+k) - c= k$, as desired.

\smallskip

\nin
$(2)$ \. Suppose that $L(y)>c+k$.  This implies that \ts $L(y)-L(x)>k$. Proceed analogously to~$(1)$.
Since \ts $h(x,y) < k$, there exists \ts $w \in X$ \ts such that \ts $L(x) < L(w) < L(y)$,
and either \ts $w\.{}||\.{}x$ \ts or \ts  $w\.{}||\.{}y$.  Assume that \ts
$w \. || \. x$, and let $w$ be such an element for which $L(w)$ is minimal.
	 Let \ts $a:=L(x)$ \ts and \ts $b:=L(w)$.
	This minimality assumption implies that every  \. $u\in \{L^{-1}(a),\ldots, L^{-1}(b-1) \}$  \ts satisfies \ts $u \succ x$,
	which gives \ts $u \. ||\. w$.

Define \ts $L'\in \Ec(P)$, obtained from $L$ by setting
$$L'(w):=L(x), \quad L'(x):=L(x)+1, \quad L'(u) = L(u)+1 \ \  \text{for all} \ \ u \in X \ \ \text{s.t.} \ \ a\le L(u)\le b-1,
$$
and setting \ts $L'(v):=L(v)$ \ts for all other elements \ts $v\in X$. Note that \ts $L'(y)=L(y)$ \ts by definition.

Denote by \. $\Psi: L \to L'$ \. the resulting map on~$\Ec(P)$. From above, $\Psi$ \ts decreases the difference
\ts $L(y)-L(x)$ \ts by one when defined.  Iterate~$\Psi$ until we obtain a
linear extension \ts $L^\circ$ \ts that satisfies \. $L^\circ(y)-L^\circ(x)= k$, as desired.

The case \ts $w\.{}||\.{}y$ \ts is completely analogous. This completes the proof of case~$(2)$, and
the  ``if'' \ts direction.
\end{proof}

\smallskip

\begin{cor}\label{cor:KS-positivity}
Let $P=(X,\prec)$ be a poset on \ts $|X|=n$ \ts elements, let \. $k>0$,
and let \ts $x,y\in X$ \ts be distinct elements. Then deciding whether \.
$\aFr(k)>0$ \. can be done in \ts {\rm poly$(n)$} \ts time.
\end{cor}

\medskip

\subsection{Complete equality conditions in the Kahn--Saks inequality}\label{ss:posets-KS}
As we discuss in the introduction, the equivalence (a)~$\Rightarrow$~(b)
in Theorem~\ref{t:q-KS-equality} does not extend to general posets.
However, the condition~(b) which states \. $\aF(k) \ts = \ts \aF(k+1) \ts = \ts \aF(k-1)$ \.
is of independent interest and perhaps can be completely characterized.
Below we give some partial results in this direction.

\smallskip

First, observe that the equality condition~(c) in
Theorem~\ref{t:Sta-equality-gen} is remarkably clean
when compared to our condition~(e) in Theorem~\ref{t:q-KS-equality}.
This suggests the following natural generalization.

\smallskip

Let \ts $P=(X,\prec)$ \ts and let \ts $x,y \in X$.
We write \. $h(x,y):= |\{u\in X\.{}:\.{}x\prec u \prec y\}|$.
We say that  \ts $(x,y)$ \ts satisfies \ts \defn{$k$-midway property},
if

\smallskip

\qquad $\circ$ \ $f(z) +g(y)\ts > \ts  n-k$ \. for every  $z \in X$ such that \ts $x \prec z$ \ts and \ts $z \not \succ y$,

\qquad $\circ$ \ $h(z,y) \ts > \ts  k$ \. for every \ts $z \prec x$, \. and \. $f(y) \ts > \ts k$.

\smallskip

\nin
Note that the last condition \. $f(y) >k$ \. is equivalent to \. $h(z,y) >k$ \.
for \. $z= \wh 0$,  i.e.\ can be dropped when the element~$\wh 0$ is added to~$P$.

\smallskip

\nin
Similarly, we say that \ts $(x,y)$ \ts
satisfies \ts \defn{dual $k$-midway property}, if:

\smallskip

\qquad $\circ$ \ $g(z)+f(x)\ts > \ts  n-k$ \. for every  $z \in X$ such that \ts $z \prec y$ \ts and \ts $z \not \precc x$,

\qquad $\circ$ \ $h(x,z) \ts > \ts  k$ \. for every \ts $z \succ y$, \. and \. $g(z) >k$.

\medskip

By definition, pair \ts $(x,y)$ \ts satisfies the $k$-midway property in the poset $P=(X,\prec)$,
if and only if pair $(y,x)$ satisfies the dual $k$-midway property in the \defn{dual poset}
\ts $P^\ast=(X,\prec^\ast)$,
obtained by reversing the partial order:  \ts $u \prec v$ \. $\Leftrightarrow$ \. $v \prec^\ast u$,
for all \ts $u,v\in X$.

\medskip

\begin{conj}[{{\it {\color{red} Complete equality condition for the Kahn--Saks inequality}}}]\label{conj:KS-equality}
Let \ts $x,y\in X$ \ts be distinct elements of a finite poset \ts $P=(X,\prec)$.
Denote by \ts $\aFr(k)$ \ts the number of linear extensions \ts $L\in \Ec(P)$,
such that \ts $L(y)-L(x)=k$.  Suppose that \ts $k \in \{2,\ldots,n-2\}$ \ts and \ts $\aFr(k)>0$. Then the following are equivalent:
\begin{enumerate}
		[{label=\textnormal{({\alph*})},
		ref=\textnormal{\alph*}}]
\item  \ $\aFr(k) \. = \.\aFr(k+1) \. = \. \aFr(k-1)$,
\item \ there is an element \ts $z\in \{x,y\}$, such that
for every \ts $L\in \Ec(P)$ \ts for which $L(y)-L(x)=k$, \\
there are elements \ts $u,v \in X$ \ts which satisfy \. $u\. || \.z$, \. $v\. || \. z$, and \.
$L(u)+1=L(z) = L(v)-1$,
\item \ the pair \ts $(x,y)$ \ts satisfies either \ts the $k$-midway \ts or \ts the dual $k$-midway property.
\end{enumerate}
\end{conj}

\medskip

Below we prove three implications, which reduce the conjecture to the implications \. (a)~$\Rightarrow$~(c).


\medskip

\begin{prop}\label{prop:equality-KS-c->b}
In the notation of Conjecture~\ref{conj:KS-equality}, we have {\rm \. (b)~$\Rightarrow$~(a).}
\end{prop}

\begin{proof}
What follows is a variation on the argument in~$\S$\ref{ss:q-KS-setup}.
Without loss of generality, assume that \ts $z=x$.  Denote \ts $\cF(i) :=\{L\in \Ec(P)~:~L(y)-L(x)=i\}$,
so that \ts $\aF(i) = |\cF(i)|$.
Condition~(b) implies that there is an injection \. $\cF(k) \to \cF(k+1)$ \. given by relabeling \ts $x \leftrightarrow u$,
so that \ts $L(y)-L(x)=k+1$ \ts  and \ts $L(u)=L(x)+1$. Thus, \ts $\aF(k) \le \aF(k+1)$.  Similarly, we obtain
\ts $\aF(k) \le \aF(k-1)$ \ts  by relabeling \ts $x \leftrightarrow v$.  However, by the Kahn--Saks
inequality (Theorem~\ref{t:KS}), we have \. $\aF(k)^2 \ts \ge \ts \aF(k-1)\. \aF(k+1)$, implying
that all inequalities are in fact  equalities.
\end{proof}

\medskip

\begin{thm}\label{t:equality-KS-dc}
In the notation of Conjecture~\ref{conj:KS-equality}, we have {\rm \. (b)~$\Leftrightarrow$~(c).}
\end{thm}

\smallskip

In other words, condition~(b) in Conjecture~\ref{conj:KS-equality}, which is the same as
condition~(e) in Theorem~\ref{t:q-KS-equality}, can be viewed as a stepping stone towards
the structural condition~(c) in the conjecture.  We omit it from the introduction
for the sake of clarity.

\smallskip

\begin{proof}[Proof of Theorem~\ref{t:equality-KS-dc}]
For \. (c)~$\Rightarrow$~(b),
let \ts $(x,y)$ \ts be a pair of elements which satisfies the $k$-midway property.
We prove~(b) by setting \ts $z \gets x$.
Let $L \in \Ec(P)$ such that \ts $L(y)-L(x)=k$.
Note that $L(x)>1$
as otherwise \ts $L(y)=k+1$, which contradicts the assumption that $f(y) >k$.
Let \ts $u \in X$ \ts be such that \ts $L(u) = L(x)-1$.
Suppose to the contrary that  \ts $u \prec x$.
It then follows from $k$-midway property that \ts $h(u,y) >  k$.
 On the other hand, since \ts $L(u) =L(x)-1 = L(y)-k-1 $, we have \ts $h(u,y) \leq k$, and gives us the desired contradiction.

Now, let \ts $v \in X$ \ts be such that \ts $L(v) = L(x)+1$.
We will again show that  \ts $v \. || \. x$.
Suppose to the contrary that $v \succ x$.
Note that \ts $v \not \succ y$ \ts since $L(v)<L(y)$.
It then follows from $k$-midway property that \ts $f(v) >  n-g(y)-k$.
 On the other hand, since $L(v) =L(x)+1$, we have \ts $f(v) \leq L(x) = L(y)-k$.
 We then obtain  \ts $L(y) - k >  n-g(y)-k$, which contradicts the fact that \ts $L(y) \leq n-g(y)$.

Thus, the pair of elements \ts $(u,v)$ \ts are as in~(b), as desired.
The case when $(x,y)$ satisfies the dual $k$-midway property leads analogously to~(b) by setting \ts $z \gets y$.

\medskip

For \. (b)~$\Rightarrow$~(c), suppose that in~(b) we have \ts $z=x$.
Now let \ts $w \in X$ \ts be such that \ts $w \succ x$ \ts and \ts  $w \not \succ y$.
 The proof is based on the following

\medskip

\nin
\textbf{Claim:}
There exists a linear extension \ts $L\in \Ec(P)$, such that
\[L(y)-L(x)=k,  \quad L(y)=n-g(y), \quad \text{ and } \quad L(w)=f(w)+1. \]

\medskip

\nin
{\it Proof of Claim.} \.
Since \ts $\aF(k)>0$, there exists a linear extension \ts $L\in \cF(k)$, i.e.\ such that \ts $L(y)-L(x)=k$.
The claim follows if \ts $L(y)=n-g(y)$ and \ts $L(w)=f(w)+1$.
So suppose that \ts $L(y) < n-g(y)$.

Then there exists \ts $p \in X$ \ts such that \ts $p \. || \. y$ \ts and \ts $L(p)>L(y)$,
and let $p$ be such an element for which $L(p)$ is minimal.  Let \ts $a:=L(y)$ \ts and \ts $b:=L(p)$.
This minimality assumption implies that every  \. $q\in \{L^{-1}(a),\ldots, L^{-1}(b-1) \}$  \ts satisfies \ts $q \succ y$,
which implies \ts $q \. ||\. p$.

Now, by~(b) there exists \ts $v \in X$ \ts  such that \ts $v \. || \. x $ \ts and \ts $L(v)=L(x)+1$.
Define \ts $L'\in \Ec(P)$ \ts by setting
\begin{align*}
& L'(p):=L(y), \quad L'(y):=L(y)+1, \quad  L'(x): =L(x)+1, \quad L'(v):=L(x), \\
 & L'(q) := L(q)+1 \ \ \ \text{for all} \ \ \ q \in X \ \ \ \text{s.t.} \ \ \. a\le L(p)\le b-1, \quad \text{and}
\end{align*}
setting \ts $L'(q):=L(q)$ \ts for all other elements \ts $q\in X$.
Note that \ts $L'(y)-L'(x)=k$, so $L'\in \cF(k)$.

Denote by \. $\Om: L \to L'$ \. the resulting map on~$\cF(k)$. From above, $\Om$ \ts
increases \ts $L(y)$ \ts by one when defined.  Iterate~$\Om$ until we obtain a
linear extension that satisfies \ts $L(y)-L(x)=k$ \ts and \ts $L(y)=n-g(y)$.

We will now show that  we can modify the current $L$ to additionally satisfy \ts $L(w)=f(w)+1$.
We are done if this is already the case, and since $L(w) \geq f(w)+1$ by definition of $f$, we can without loss of generality assume that \ts $L(w)> f(w)+1$. We will find a new $L'$, which preserves $L(x)$ and $L(y)$ while decreasing $L(w)$ by 1.
 Note that \ts $L(x) < L(w) < L(y)$ \ts since  \ts $x \prec w \not \succ y$ and $L(y)$ is at its maximal value.
Since $L(w)>f(w)+1$, there exists $p \in X$ such that \ts $p \. || \. w$ \ts and $L(p) < L(w)$, and let $p$ be such an element for which $L(p)$ is maximal.
By the same argument as in the previous paragraph, we can then create a new linear extension~$L'$ by moving~$p$ to the right of $w$, i.e.,
\[ L'(p)  \ := \ L(w), \qquad L'(w)=L(w)-1, \qquad  L'(q) \ := \
  \begin{cases}
  	L(q)- 1 & \text{ if } \   L(p)< L(q) < L(w) \\
  	L(q) & \text{if }  \  L(q) > L(w) \ \text{ or } \ L(q) <L(p).
  \end{cases}  \]
Note  that \ts $L'(w)=L(y)$ \ts  since  \ts $L(p)$ \ts and \ts $L(w)$ \ts are less than $L(y)$ by assumption.
If \ts $L(x)<L(p)$, then \ts $L'(x)=L(x)$ and we are done.
Otherwise, let $v$ be the element of $P$ such that \ts $L(v)=L(x)+1$ \ts and \ts $v \. || \. x$, which exists by~(b).
Now exchange the values of $L'$ at $v$ and $x$, so \ts $L'(v)=L(x)+1$ and \ts $L'(x)=L(v)-1=L(x)$.
This is so that the resulting linear extension $L'$ always satisfies \ts $L'(y)-L'(x)=k$.
Also note that \ts $L'(w)=L(w)-1$ \ts and \ts $L'(y)=L(y)$ \ts by construction.

We thus  obtain a map \ts $\Theta: L \to L'$ \ts  such that $L'(w)=L(w)-1$,
while preserving the values of the linear extensions at $y$ and $x$, i.e. \ts $L'(y)=n-g(y)$ \ts and \ts $L'(y)-L'(x)=k$.
Iterate~$\Theta$ until we obtain a linear extensions that satisfies \ts $L(w)=f(w)+1$,
and the proof is complete.
\ $\sq$

\medskip

We now have
$$
f(w)+1 \ = \ L(w) \ > \  L(x) +1 \ =  \ n- g(y)-k+1,
$$
where the equalities are due to the claim above, while the inequality is due to applying~(b) to conclude that $L(w) \neq L(x)+1$ since $w \neq v$.
This shows that \ts $f(w) > n-g(y)+k$.

	Now let $w \in X$ such that $w \prec x$.
By an analogous argument, we conclude that
there exists a linear extension \ts $L\in \Ec(P)$ \ts such that \ts $L(y)-L(x)=k$, \ts $L(y)=f(y)+1$, and $L(y)-L(w)= h(w,y)+1$.
On the other hand, we have
\[  h(w,y)+ 1  \ = \ L(y) - L(w)  \ = \ k + L(x) -L(w) \ \geq  \ k+2,  \]
where the equalities are due to the claim above, while the inequality is due to applying~(b) to $w$.
This shows that \. $h(w,y) \geq k+1$.
Now note that  \. $L(x) >1$ \. by (b), so it then follow that
\[ f(y)  \ = \  L(y) -1 \ =  L(x) + k- 1 > k. \]
We thus conclude that  \ts $(x,y)$ \ts satisfies
the \ts $k$-midway property.

\smallskip

Finally, suppose that \ts $z=y$. In this case we obtain that \ts $(x,y)$ \ts
satisfies the dual $k$-midway property.  This follows by taking a
dual poset~$P^\ast$, and relabeling \ts $x\leftrightarrow y$, \ts $f \leftrightarrow g$  \ts
in the argument above. This completes the proof of the theorem.  \end{proof}

\smallskip

\begin{rem}{\rm
Our proof of the \. (c)~$\Rightarrow$~(b) \. implication in Theorem~\ref{t:equality-KS-dc},
is a variation on the proof of the implication \. (c)~$\Rightarrow$~(b) \.
in Theorem~\ref{t:Sta-equality-gen}, given in~\cite[$\S$15.1]{SvH}.  Of course, the
details are quite a bit more involved in our case.
}
\end{rem}

\medskip

\subsection{Back to posets of width two}
For posets of width two,
the $k$-midway property is especially simple,
and can be best understood from Figure~\ref{fig:separable}.

\smallskip

\begin{prop} \label{prop:equality-KS-midway}
In notation of Conjecture~\ref{conj:KS-equality}, let \ts $P=(X,\prec)$ \ts be a poset
of width two, and let \ts $(\Cr_1,\Cr_2)$ \ts be partition into two chain as in the
introduction.  Let \ts $(x,y)$ \ts
be a pair of elements in~$\Cr_1$, where \ts $x=\al_s$ \ts
and \ts $y=\al_{s+r}$.
  Then \ts $(x,y)$ \ts satisfies $k$-midway property if and only if
there are integers \. $1< c < d\le n$, such that:

\smallskip

$\circ$ \ $\al_{s-1}\prec \be_{c-s} \prec \ldots \prec \be_{d-s} \prec \al_{s+1}$\.,

$\circ$ \ $\be_{c+k-r-s}\prec \al_{s+r} \prec \be_{d+k-r-s}$\.,

$\circ$ \ $\al_s \.{}||\.{}\be_{c-s}\. ,\. \ldots\.,\. \al_s \.{}||\.{}\be_{d-s}$\., and

$\circ$ \ $\al_{s+r} \.{}||\.{}\be_{c+(k-r-s)+1}\.,\. \ldots\ts,\ts \al_{s+r} \.{}||\.{}\be_{d+(k-r-s)-1}$\..

\end{prop}

\smallskip

The proposition follow directly from the proof of Theorem~\ref{t:q-KS-equality} in
Section~\ref{s:q-KS-equality}, where we let \ts $c:= \vmin$ \ts and \ts $d:=\vmax+1$\ts.
We omit the details.
Note also that when \ts $y=\wh 1$ is the maximal element, we obtain the $(n-k)$-pentagon property.

\begin{figure}[hbt]
\begin{center}
\includegraphics[width=7.5cm]{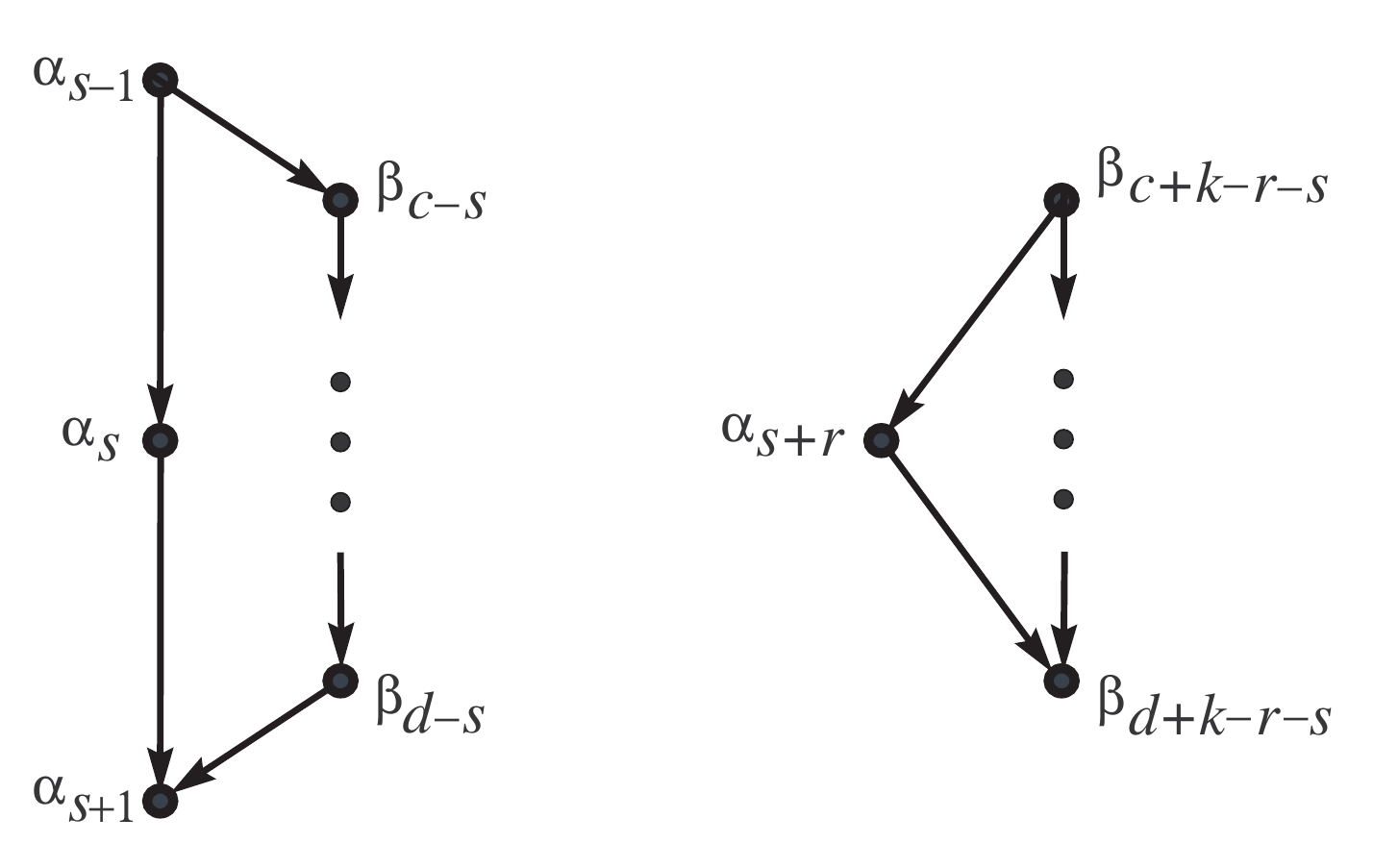}
\end{center}
\vskip-.3cm
\caption{The $k$-midway property for the pair \ts $(\al_s, \al_{s+r})$.
The arrows point from smaller to larger poset elements.}\label{fig:separable}
\vskip.3cm
\end{figure}

\smallskip

\begin{rem}{\rm Figure~\ref{fig:separable} may seem surprising at first due
to its vertical symmetry.  So let us emphasize that in contrast with the $k$-pentagon
property, the $k$-midway property is not invariant under poset duality due to the asymmetry
of the labels. This is why it is different from the dual $k$-midway property
even for posets of width two.
}\end{rem}


\bigskip

\section{Final remarks and open problems}\label{sec:finrem}

\subsection{}\label{ss:finrem-hist}
Finding the equality conditions is an important problem for inequalities
across mathematics, see e.g.~\cite{BB}, and throughout the sciences,
see e.g.~\cite{Dahl}.  Notably, for geometric inequalities,
such as the \emph{isoperimetric inequalities}, these problems are classical
(see e.g.~\cite{BZ}), and in many cases the equality conditions are equally
important and are substantially harder to prove. For example, in the
\emph{Brunn--Minkowski inequality}, the equality conditions
are crucially used in the proof of the \emph{Minkowski theorem}
on existence of a polytope with given normals
and facet volumes (see e.g.\ $\S$7.7, $\S$36.1 and $\S$41.6 in~\cite{Pak}).

For poset inequalities, the equality conditions have also been studied, see
e.g.\ an overview in~\cite{Win}.  In fact, Stanley's original paper~\cite{Sta}
raises several versions of this question. In recent years, there were a number of key
advances on combinatorial inequalities using algebraic and analytic tools, see e.g.~\cite{CP,Huh},
but the corresponding equality conditions are understood in only very few instances.\footnote{See
a {\tt MathOverflow} discussion here:  \url{https://mathoverflow.net/questions/391670}.}

\medskip

\subsection{}\label{ss:finrem-KS}
In a special case of the Kahn--Saks inequality, finding the equality conditions
in full generality remains a major challenge.  From this point of view, the
equivalences \. (a) \. $\Rightarrow$ \. (b)  \. $\Rightarrow$ \. (c) \. in
Theorem~\ref{t:q-KS-equality} combined with Theorem~\ref{t:equality-KS-dc}
is the complete characterization in a special case of width two posets with
two elements in the same chain.  As we mentioned in the introduction, this
result is optimal and does not extend even to elements in different chains.

We should also mention that both Stanley's inequality (Theorem~\ref{t:Sta})
and the equality conditions in Stanley's inequality
inequality (Theorem~\ref{t:Sta-equality-gen}) was recently proved by elementary
means in~\cite{CP}. Despite our best efforts, the technology of~\cite{CP} does
not seem to translate to the Kahn--Saks inequality, suggesting the difference
between the two.  In fact, the close connection between the inequalities and
equality conditions in the proofs of~\cite{CP} hints that perhaps the equality
conditions of the Kahn--Saks inequality are substantially harder to obtain.

\medskip

\subsection{}\label{ss:finrem-ineq}  From the universe of poset inequalities,
let us single out the celebrated \emph{XYZ inequality}, which was later proved to
be always strict \cite{Fish} (see also~\cite{Win}).  Another notable example
in the \emph{Ahlswede--Daykin correlation inequality} whose equality was
studied in a series of papers, see~\cite{AK} and references therein.

The \emph{Sidorenko inequality} is an equality if and only if a poset is
\emph{series--parallel}, as proved
in the original paper~\cite{Sid}.  The latter inequality turned out to be a
special case of the conjectural \emph{Mahler inequality}.  It would be interesting to
find an equality condition of the more general \emph{mixed Sidorenko inequality}
for pairs of two-dimensional posets, recently introduced in~\cite{AASS}.

In our previous paper~\cite{CPP2}, we proved both the \emph{cross--product inequality}
as well the equality conditions for the case of posets of width two.  While in full
generality this inequality implies the Kahn--Saks inequality, the reduction does not
preserve the width of the posets, so the results in~\cite{CPP2} do not imply the results in
this paper.
Let us also mention some recent work on poset inequalities for posets of width
two~\cite{Chen,Sah} generalizing the classical approach in~\cite{Lin}.

\medskip

\subsection{}\label{ss:finrem-region}
The bijection in Lemma~\ref{l:interpretation lattice path}, see also
Remark~\ref{rem:general-regions}, is natural
from both order theory and enumerative combinatorics points of view.
Indeed, the order ideals of a width two poset $P$ with fixed chain partitions
\ts $(\Cr_1,\Cr_2)$ \ts are in natural bijection with lattice points in a
region \ts $\Reg(P)\ssu \zz^2$.
Now the \emph{fundamental theorem for finite distributive lattices}
(see e.g.~\cite[Thm 3.4.1]{EC}), gives the same bijection between \ts $\Ec(P)$ \ts
and \ts lattice paths \ts $\zero \to (\ana,\bnb)$ \ts in \ts $\Reg(P)$.

\medskip

\subsection{}\label{ss:finrem-CFG}
As we mentioned in the introduction, the injective proof of the Stanley
inequality (Theorem~\ref{t:q-Sta}) given in Section~\ref{sec:Sta},
does in fact coincide with the CFG injection given in~\cite{CFG}.
The latter is stated somewhat informally, but we find the formalism
useful for generalizations.  In a different direction, our breakdown
into lemmas allowed us a completely different generalization of the Stanley
inequality to exit probabilities of random walks, which we discuss in a
follow up paper~\cite{CPP3}.

\medskip

\subsection{}\label{ss:finrem-promo}
Maps \ts $\Phi,\Psi,\Om, \Theta$ \ts on \ts $\Ec(P)$ \ts used in the proofs of
Theorem~\ref{t:F-positive} and Theorem~\ref{t:equality-KS-dc}, are closely
related to the \emph{promotion} map heavily studied in poset literature,
see e.g.~\cite[$\S$3.20]{EC} and~\cite{Sta-promo}.  We chose to avoid
using the known properties of promotion to keep proofs simple and
self-contained.  Note that the promotion map can also be used to
prove Proposition~\ref{prop:Sta-zero} as we do in greater generality 
in the forthcoming~\cite{CPP4}.

\medskip

\subsection{}\label{ss:finrem-CS}
Recall that computing \ts $e(P)$ \ts is $\SP$-complete even for posets of
height two, or of dimension two; see~\cite{DP} for an overview.
The same holds for \ts $\aN(k)$, which both a refinement and a generalization
of~$\ts e(P)$.  Following the approach in~\cite{Pak}, it is natural to conjecture that \.
$\rT(x,k)\. :=\. \aN(k)^2\ts -\ts \aN(k+1)\.\aN(k-1)$ \. is $\SP$-hard for general posets.
We also conjecture that \ts $\rT(x,k)$ \ts is not in~$\SP$ even though it
is in \ts $\GapP_{\ge 0}$ \ts by definition.  From this point of view,
Corollary~\ref{cor:Sta-positivity} is saying that the decision problem
whether \ts $\rT(x,k)=0$ \ts is in~$\ts \poly$, further complicating the matter.

\medskip

\subsection{}\label{ss:finrem-mixed}
There is an indirect way to derive both Corollaries~\ref{cor:Sta-positivity}
and~\ref{cor:KS-positivity} without explicit combinatorial conditions for
vanishing of \ts $\aN(k)$ \ts and $\aF(k)$, given in
Proposition~\ref{prop:Sta-zero} and Theorem~\ref{t:F-positive}, respectively.
In fact, the vanishing problem is in \ts $\poly$ \ts by the following general result:

\smallskip

\begin{thm} \label{t:mixed-zero}
Let \ts $P=(X,\prec)$ \ts be a finite poset with \ts $|X|=n$ \ts elements, let \.
$x_1\ldots,x_k \in X$ \. be distinct poset elements, and let
\. $a_1,\ldots,a_k\in \{1,\ldots,n\}$ \. be distinct integers.
Finally, let \ts $\aNr(a_1,\ldots,a_k)$ \ts be the number of linear extensions \ts $L \in \Ec(P)$ \ts
such that \ts $L(x_i)=a_i$ \ts for all \ts $1\le i \le i$.
Then, deciding whether \. $\aNr(a_1,\ldots,a_k)=0$ \. can be done
in \ts {\rm poly$(n)$} \ts time.
\end{thm}


\begin{proof}
It was shown by Stanley~\cite[Thm~3.2]{Sta}, that \ts $\aN(a_1,\ldots,a_k)/(n-k)!$ \ts
is equal to the \emph{mixed volume} of certain polytopes~$\ts K_i$ \ts given by
explicit combinatorial inequalities.  In the terminology of~\cite[p.~364]{DGH}, these
polytopes~$\ts K_i$ \ts are \emph{well-presented}, so by \cite[Thm~8]{DGH}
the vanishing of the mixed volume can be decided in polynomial time.
\end{proof}

\smallskip

To see the connection between the theorem and condition \ts $\aF(k)=0$, note
that for every fixed \ts $x, y \in X$, we have \. $\aF(k)= \aN(1,k+1) + \ldots + \aN(n-k,n)$.
In the same way, the vanishing \ts $\aF(k,\ell) =0$ \ts in the cross--product inequality
(see~\cite{CPP2}) can also be decided in polynomial time.

Note that the proof in \cite[Thm~8]{DGH} involves a classical
but technically involved \emph{matroid intersection algorithm} by Edmonds (1970).
Thus, finding an explicit combinatorial condition for vanishing
of \. $\aN(a_1,\ldots,a_k)$ \. is of independent interest.\footnote{Most recently,
the authors were able to obtain such conditions by a technical algebraic argument, 
see~\cite{CPP4}.}

\vskip.5cm
	
\subsection*{Acknowledgements}
We are grateful to Fedya Petrov, Ashwin Sah, Raman Sanyal, Yair Shenfeld and Ramon van Handel
for helpful discussions on the subject, and to Leonid Gurvits for telling us about the
\cite{DGH} reference.  Special thanks to Ramon van Handel and Alan Yan for pointing out
an important error in the previous version of the paper, and for suggesting Example~\ref{ex:Ramon}.
The last two authors were partially supported by the NSF.

\vskip1.2cm


\

\vskip.5cm	

\end{document}